\documentclass[11pt]{amsart}
\usepackage{amssymb,amsbsy, amsthm, amsmath, amstext, amsopn}
\usepackage{aliascnt}
\usepackage{xcolor}
\usepackage{microtype}
\usepackage[centering,margin=1.2in]{geometry}
\usepackage[colorlinks,
            linkcolor=black!50!red,
            citecolor=blue,
            pdfpagemode=None]{hyperref}

\numberwithin{equation}{section}

\theoremstyle{plain}

\newaliascnt{theorem}{equation}
\newtheorem{thm}[theorem]{Theorem}
\aliascntresetthe{theorem}

\newaliascnt{prop}{equation}
\newtheorem{prop}[prop]{Proposition}
\aliascntresetthe{prop}

\newaliascnt{lemma}{equation}
\newtheorem{lemma}[lemma]{Lemma}
\aliascntresetthe{lemma}

\newaliascnt{corollary}{equation}
\newtheorem{corollary}[corollary]{Corollary}
\aliascntresetthe{corollary}

\theoremstyle{definition}

\newaliascnt{remark}{equation}
\newtheorem{remark}[remark]{Remark}
\aliascntresetthe{remark}

\newaliascnt{defn}{equation}
\newtheorem{defn}[defn]{Definition}
\aliascntresetthe{defn}

\usepackage{cleveref}

\crefname{theorem}{Theorem}{Theorems}
\crefname{prop}{Proposition}{Propositions}
\crefname{section}{Section}{Sections}
\crefname{subsection}{Section}{Sections}
\crefname{defn}{Definition}{Definitions}
\crefname{remark}{Remark}{Remarks}
\crefname{lemma}{Lemma}{Lemmas}
\crefname{corollary}{Corollary}{Corollaries}

\newcommand{\mb}[1]{\mathbf{#1}}

\DeclareMathOperator{\Ad}{Ad}
\DeclareMathOperator{\tr}{tr}
\DeclareMathOperator{\Sym}{Sym}

\newcommand{\wt}{\widetilde}
\newcommand{\wh}{\widehat}
\newcommand{\ot}{\otimes}
\newcommand{\ol}[1]{\overline{#1}}
\newcommand{\comments}[1]{}

\def\al{\alpha}
\def\be{\beta}
\def\de{\delta}
\def\De{\Delta}

\def\io{\iota}

\def\la{\lambda}
\def\th{\theta}
\def\si{\sigma}

\def\om{\omega}
\def\Om{\Omega}
\def\ep{\epsilon}
\def\Th{\Theta}
\def\Ga{\Gamma}

\def\Bb{{\mathcal B}}

\def\Dd{{\mathcal D}}

\def\Gg{{\mathcal G}}

\def\Ll{{\mathcal L}}

\def\Pp{{\mathcal P}}
\def\Qq{{\mathcal Q}}

\def\Uu{{\mathcal U}}

\def\CC{{\mathbb C}}

\def\NN{{\mathbb N}}

\def\ZZ{{\mathbb Z}}

\def\bb{{\mathfrak b}}

\def\dd{{\mathfrak d}}

\newcommand{\fg}{\mathfrak g}
\def\hh{{\mathfrak h}}

\def\ll{{\mathfrak l}}

\def\nn{{\mathfrak n}}

\def\ss{{\mathfrak s}}

\def\sl{\ss\ll}

\begin{document}

\title[Double Bruhat Cells and Integrable Systems]{Double Bruhat Cells in Kac-Moody Groups and Integrable Systems}


\author{Harold Williams}
\address{Harold Williams\newline
University of California, Berkeley\newline
Department of Mathematics\newline
Berkeley CA 94720\newline
USA}
\email{harold@math.berkeley.edu}

\begin{abstract}

We construct a family of integrable Hamiltonian systems generalizing the relativistic periodic Toda lattice, which is recovered as a special case.  The phase spaces of these systems are double Bruhat cells corresponding to pairs of Coxeter elements in the affine Weyl group.  In the process we extend various results on double Bruhat cells in simple algebraic groups to the setting of Kac-Moody groups.  We also generalize some fundamental results in Poisson-Lie theory to the setting of ind-algebraic groups, which is of interest beyond our immediate applications to integrable systems.

\end{abstract}

\maketitle

\section{Introduction}

This paper describes a class of completely integrable Hamiltonian systems generalizing the relativistic periodic Toda lattice, introduced in \cite{Ruijsenaars1990}.  We identify the phase space of this particular system with a double Bruhat cell of the $A_n^{(1)}$ affine Kac-Moody group, and its Hamiltonians with restrictions of invariant functions.   This refines the well-known observation that it admits a Lax form which is Hamiltonian with respect to the Poisson-Lie bracket induced by the trigonometric $r$-matrix \cite{Suris1991}.  A larger family of systems can then be obtained by transporting the construction to other double Bruhat cells and other groups.  On a general double Bruhat cell the invariant functions will not necessarily restrict to a maximal set of Poisson-commuting functions, but we show that a sufficient condition for this is that the cell correspond to a pair of Coxeter elements in the affine Weyl group.  This construction generalizes that of \cite{Hoffmann2000}, which treated semisimple algebraic groups and where the term Coxeter-Toda lattice was introduced for the resulting systems.   

The double Bruhat cells of a semisimple algebraic group are fundamental objects in Poisson-Lie theory, total positivity, and the theory of cluster algebras \cite{Fomin1999,Berenstein2005}.  Our construction requires an extension of various results on these cells to the setting of Kac-Moody groups.  In particular, we show that the double Bruhat cells of a symmetrizable Kac-Moody group are smooth finite-dimensional Poisson varieties equipped with distinguished factorization coordinates generalizing those of \cite{Fomin1999}.  

One of the main challenges in this extension is the absence of an adequate foundation for the Poisson-Lie theory of ind-algebraic groups.  Accordingly, we include a self-contained treatment of the necessary infinite-dimensional Poisson-Lie theory, which is of interest beyond our immediate application.  Though Poisson brackets on loop groups have a long history in mathematical physics, they are often dealt with less precisely than their finite-dimensional counterparts.  Our results provide one general framework in which they can be treated rigorously.

Another family of generalized relativistic Toda systems was introduced recently in \cite{Eager2011}.  These systems are constructed from certain periodic dimer models, following \cite{Goncharov2011}.  Their phase spaces are moduli spaces of flat line bundles on a bipartite graph, and their Hamiltonians are derived from the dimer partition function.  As shown in \cite{Fock2012}, this provides a complementary description of the systems we construct in type $A_n^{(1)}$.  We note in passing that applications of relativistic Toda systems to gauge theory (as in \cite{Nekrasov1998}) are the principal motivation for their study in \cite{Eager2011}.

The connection of relativistic Toda systems with double Bruhat cells lets us exploit the combinatorial structure of the latter in a number of ways.  For example, we introduce and make essential use of the affine version of the factorization coordinates introduced in the study of total positivity \cite{Fomin1999}.  These coordinates are closely related with the cluster algebra structure on the coordinate ring of the double Bruhat cell \cite{Williams}, and are crucial for understanding the relationship between our systems and those constructed from dimer models \cite{Marshakov2012}.   Moreover, we will see that total positivity provides the natural link between our complex-algebraic construction and the usual real form of the relativistic Toda system.  Finally, the theory of cluster algebras provides a natural setting for the study of discrete symmetries of integrable systems, for example as worked out in detail for the $GL_n$ Coxeter-Toda systems in \cite{Gekhtman2011}. 

The layout of the paper is as follows.  In \cref{sec:kmgroups} we recall the needed background on affine Kac-Moody algebras and groups.  \Cref{sec:plgroups} is devoted to the Poisson-Lie theory of ind-algebraic groups, in particular symmetrizable Kac-Moody groups and their standard Poisson structure.  In \Cref{sec:dbc} we prove some geometric results about double Bruhat cells in Kac-Moody groups, and in particular describe their factorization coordinates and Poisson brackets.  Finally, in \cref{sec:acts} we show that the reduced Coxeter double Bruhat cells of an affine Kac-Moody group possess canonical integrable systems, and derive the relativistic Toda lattice from this point of view.  

\textsc{Acknowledgments.} 
I would like to thank Nicolai Reshetikhin for his generous support and for suggesting the topic of this publication.  I also thank Kevin Schaeffer, Theo Johnson-Freyd, Ed Frenkel, David Kazhdan, and Lauren Williams for valuable discussions and comments.  This research was supported by NSF grant DMS-0901431 and the Centre for Quantum Geometry of Moduli Spaces at Aarhus University.

\section{Background on Kac-Moody Algebras and Groups}\label{sec:kmgroups}

In this section we recall the needed background on Kac-Moody algebras and groups, paying particular attention to the affine case \cite{Kac1994,Kumar2002,Mathieu1988}.

\subsection{Kac-Moody Algebras} A generalized Cartan matrix $C$ is an $r \times r$ integer matrix such that
\begin{enumerate}
\item $C_{ii} = 2$ for all $1 \leq i \leq r$
\item $C_{ij} \leq 0$ for $i \neq j$
\item $C_{ij} = 0$ if and only if $C_{ji} = 0$.
\end{enumerate}
To the matrix $C$ is associated a Lie algebra $\fg:=\fg(C)$.  The Cartan subalgebra $\hh \subset \fg$ contains simple coroots $\{ h_1, \dots , h_r\}$, its dual contains simple roots $\{ \al_1, \dots, \al_r\}$, and these satisfy $\langle \al_j | h_i \rangle = C_{ij}$.  The algebra $\fg$ is generated by $\hh$ and the Chevalley generators $\{ e_1, f_1, \dots, e_r, f_r\}$, subject to the relations
\begin{enumerate}
\item $[h,h'] = 0$ for all $h, h' \in \hh$
\item $[h,e_i] = \langle \al_i | h \rangle e_i$, $[h,f_i] = - \langle \al_i | h \rangle f_i$ for all $h \in \hh$
\item $[e_i,f_i] = h_i$
\item $[e_i,f_j] = ad(e_i)^{1-C_{ij}}e_j = ad(f_i)^{1-C_{ij}}f_j = 0$ for all $i \neq j$.
\end{enumerate}

We assume throughout that $C$ is symmetrizable; that is, there exist positive numbers $d_i$ such that $d_iC_{ij} = d_jC_{ji}$ for all $i,j$.  In this case there is a corresponding symmetric nondegenerate invariant bilinear form on $\fg$.  It restricts nondegenerately to $\hh$, and may be normalized so that $\frac{\|\al_i\|}{2} = d_i$.

The roots of $\fg$ are the elements $\al \in \hh^*$ such that
\[
\fg_{\al} =  \{X \in \fg\: |\: [h,X] = \langle \al|h \rangle X \mathrm{\:\: for\: all\:\:} h \in \hh \}
\]
is nonzero.  Any nonzero root is a sum of simple roots with either all positive or all negative integer coefficients, and we say it is positive or negative accordingly.  We then have subalgebras
\[
\nn_+ = \bigoplus_{\al > 0} \fg_{\al}, \quad \nn_- = \bigoplus_{\al < 0} \fg_{\al}.
\]
If $\fg'$ denotes the derived subalgebra of $\fg$ and $\hh' = \bigoplus_{i=1}^r \CC h_i$, then we have vector space decompositions
\[
\fg = \nn_- \oplus \hh \oplus \nn_+, \quad \fg' = \nn_- \oplus \hh' \oplus \nn_+.
\]

The Weyl group $W$ of $\fg$ is the subgroup of $\mathrm{Aut}(\hh^*)$ generated by the simple reflections
\[
s_i: \be \mapsto \be - \langle \be|h_i \rangle \al_i.
\]
A nonzero root is said to be real if it is conjugate to a simple root under $W$, and imaginary otherwise.  A reduced word for an element of $W$ is an expression $w = s_{i_1} \cdots s_{i_n}$ such that $n$ is as small as possible; the length $\ell(w)$ is then defined as the length of such a reduced word.  

The set of dominant integral weights is $P_+ := \{\la \in \hh^* : \langle \la | h_i \rangle \geq 0 \text{ for all } 1 \leq i \leq r\}$.  For each $\la \in P_+$ there is an irreducible $\fg$-representation $L(\la)$ with highest weight $\la$, unique up to isomorphism.  The representation $L(\la)$ is the direct sum of finite-dimensional $\hh$-weight spaces, and its graded dual $L(\la)^\vee$ is an irreducible lowest-weight representation. 

We say $\fg(C)$ is of finite type if $C$ is positive definite, and affine type if $C$ is positive semidefinite.  In the former case it is a finite-dimensional semisimple Lie algebra, while in the latter it admits an alternative description in terms of loop algebras.  

More precisely, let $\fg(C)$ be a semisimple Lie algebra with Cartan matrix $C$.  Its loop algebra $L\fg := \fg(C) \otimes \CC[z^{\pm 1}]$ has a universal central extension $\widetilde{L\fg} := \CC c \oplus L\fg$ with bracket
\[
[Xz^m + Ac, Yz^n + Bc] = [X,Y]z^{m+n} + \de_{m+n,0} \langle X, Y \rangle c.
\]
The action of $\frac{d}{dz}$ on $L\fg$ by derivations extends to an action on $\wt{L\fg}$, so we have the semidirect product $\wh{L\fg} := \CC\frac{d}{dz} \ltimes \wt{L\fg}$. There is an extended Cartan matrix $\wt{C}$ such that $\wh{L\fg} \cong \fg(\wt{C})$ and $\wt{L\fg} \cong \fg'(\wt{C})$.  To form $\wt{C}$ we adjoin an extra row and column to $C$ by setting
\[
C_{0,0} = 2, \quad C_{k,0} = -\th(h_k), \quad \mathrm{and} \quad C_{0,i} = -\al_i(h_\th).
\]
Here $\th = \sum_{i=1}^r \th_i \al_i$ is the highest root of $\fg(C)$, and we will always normalize the form on $\fg(C)$ so that $\langle \th, \th \rangle = 2$ (to simplify later formulas we will also use the convention $\th_0 = 1$).  Note that we index the simple roots of a general Kac-Moody algebra by $\{1,\dots,r\}$, while we index affine simple roots by $\{0,\dots,r\}$.  Every affine Kac-Moody algebra is either of the form $\wh{L\fg}$ or a twisted version thereof; for simplicity we will only consider the former case.

\subsection{Kac-Moody Groups}\label{subsec:KMgroups}

To a generalized Cartan matrix $C$ we may also associate a group $G$, which is a simply-connected complex algebraic group when $C$ is of finite type \cite{Peterson1983,Kumar2002}.  In general $G$ is an ind-algebraic group, and shares many important properties with the simple algebraic groups, in particular a Bruhat decomposition and generalized Gaussian factorization.

For each real root $\al$, $\Gg$ contains a one-parameter subgroup $x_{\al}(t)$, and is generated by these together with the Cartan subgroup $H$ (for simple roots, we will write $x_{\pm i}(t) := x_{\pm \al_i}(t)$).  We denote the subgroups generated by the positive and negative real root subgroups by $\Uu_+$ and $\Uu_-$, respectively, and we also have the positive and negative Borel subgroups $\Bb_{\pm} := H \ltimes \Uu_{\pm}$.  If $N(H)$ is the normalizer of $H$ in $\Gg$, then $N(H)/H$ is isomorphic with the Weyl group.  In particular, the simple reflections $s_{\al}$ have representatives in $\Gg$ of the form
\begin{equation}\label{eqn:simplerootformula}
\ol{s}_{\al} = x_{\al}(1) x_{-\al}(-1) x_{\al}(1).
\end{equation}  

Recall that an ind-variety $X$ is the union of an increasing sequence of finite-dimensional varieties $X_n$ whose inclusions $X_n \hookrightarrow X_{n+1}$ are closed embeddings \cite{Shafarevich1982}.  We say a map $X \xrightarrow{\phi} Y$ of ind-varieties is regular if for all $i \in \NN$ there exists an $n(i)$ such that $\phi(X_i) \subset Y_{n(i)}$ and the restrictions $X_i \xrightarrow{\phi|_{X_i}} Y_{n(i)}$ are regular.  If the $X_n$ are affine, the coordinate ring of $X$ is
\[
\CC[X] = \varprojlim \CC[X_n],
\]
topologized as an inverse limit of discrete vector spaces; regular maps of affine ind-varieties induce continuous homomorphisms between their coordinate rings.  We can also form products of ind-varieties in the obvious way. 

\begin{defn} 
An ind-algebraic group (or ind-group) $X$ is an ind-variety with a regular group operation $X \times X \to X$.
\end{defn} 

To define the ind-group structure on $\Gg$, consider the integrable $\fg$-representation
\[
V = \bigoplus_{i=1}^{\mathrm{dim}(H)} (L(\om_i) \oplus L(\om_i)^{\vee}).
\]
Here the $\om_i$ are a $\ZZ$-basis of $\mathrm{Hom}(H,\CC^*) \subset \hh^*$ such that $\langle \om_i | h_j \rangle = \de_{i,j}$ for $1\leq i \leq r$.  The group $\Gg$ acts on integrable highest weight representations of $\fg$ and their restricted duals, hence on $V$. If $v_i$ and $v_i^{\vee}$ are the highest and lowest weight vectors of $L(\om_i)$ and $L(\om_i)^{\vee}$, respectively, the map $g \mapsto g \cdot \sum_{i=1}^r(v_i + v_i^{\vee})$ embeds $\Gg$ injectively into $V$.  We may filter $V$ by finite direct sums of its weight spaces, and the intersections of $\Gg$ with these are closed subvarieties that define an ind-group structure on $\Gg$ \cite[7.4.14]{Kumar2002}.  The subgroups $H$, $\Uu_{\pm}$, and $\Bb_{\pm}$ are then closed subgroups.

\begin{prop} \label{prop:gaussian}
\emph{(\cite[6.5.8 and 7.4.11]{Kumar2002})} The multiplication map $\Uu_- \times H \times \Uu_+ \to \Gg$ is a biregular isomorphism onto an open subvariety $\Gg_0$.  Thus for any $g \in \Gg_0$ we may write
\[
g = [g]_- [g]_0 [g]_+
\]
for some unique $[g]_{\pm} \in \Uu_{\pm}$ and $[g]_0 \in H$.  Moreover, the maps 
\[
\Gg_0 \to \Uu_{\pm} \: (\text{resp. } H) \quad g \mapsto [g]_{\pm} \: (\text{resp. } [g]_0)
\] are regular.
\end{prop}

\begin{prop} \label{prop:Bruhat}
\emph{(\cite[7.4.2]{Kumar2002})} The double coset decomposition of $\Gg$ with respect to $\Bb_{\pm}$ can be written as
\[
\Gg = \bigsqcup_{w \in W} \Bb_+ \dot{w} \Bb_+ = \bigsqcup_{w \in W} \Bb_- \dot{w} \Bb_-.
\]
Here $\dot{w}$ is any representative for $w$ in $\Gg$.  In particular, $\Gg$ is a disjoint union of the double Bruhat cells
\[
\Gg^{u,v} := \Bb_+ \dot{u} \Bb_+ \cap \Bb_- \dot{v} \Bb_-.
\]
\end{prop}

For any $w \in W$ we have closed subgroups
\begin{gather*}
\Uu_{\pm}(w) := \Uu_{\pm} \cap \dot{w}^{-1} \Uu_{\mp} \dot{w}, \quad \Uu'_\pm(w) := \Uu_\pm \cap \dot{w}^{-1} \Uu_{\pm} \dot{w}.
\end{gather*}
The $\Uu_\pm(w)$ are $\ell(w)$-dimensional unipotent groups.  As above, $\dot{w}$ is some representative of $w$ in $\Gg$, but the resulting subgroup is independent of this choice. 

\begin{prop}\label{prop:unipotentfactorization}
\emph{(\cite[6.1.3]{Kumar2002})} For any $w \in W$, the multiplication maps
\[
\Uu_\pm(w) \times \Uu'_\pm(w) \to \Uu_\pm
\]
are biregular isomorphisms.
\end{prop}
\begin{proof}
That these are bijections follows from \cite[6.1.3]{Kumar2002}.  The inverse map is regular since the projection maps from $\Uu_\pm$ to $\Uu_\pm(w)$, $\Uu'_\pm(w)$ are: we can write them as conjugation by $\dot{w}$ followed by the maps $g \mapsto [g]_{\pm}$ of \cref{prop:gaussian}.
\end{proof}

The Bruhat decomposition then admits the following refinement:

\begin{corollary} \label{prop:refinedBruhat}
The natural maps
\[
\Uu_\pm \to \Uu_\pm(w) \dot{w} \Bb_\pm / \Bb_\pm, \quad g \mapsto g\dot{w}\Bb_\pm
\]
are biregular isomorphisms.  In particular, the Bruhat cells can be written as
\[
\Bb_\pm \dot{w} \Bb_\pm = \Uu_\pm(w) \dot{w} \Bb_\pm.
\]
\end{corollary}

For each simple root $\al$, $\Gg'$ has a corresponding $SL_2$ subgroup $G_{\al}$ generated by $x_{\pm \al}(t)$.  In \cref{thm:rmatrix} we will use the following observation:

\begin{prop}\label{prop:simplegenerators}
$\Gg'$ is generated by the simple root $SL_2$ subgroups $G_\al$.\footnote{Since $\Gg'$ is infinite-dimensional it does not suffice to observe that the Lie algebras of the $G_\al$ together generate $\fg$.  For example, the Lie algebra of $\Uu_+ \subset \wt{LSL}_2$ is generated by the two simple positive root spaces, yet $\Uu_+$ is not generated by \emph{any} proper subcollection of the 1-parameter positive root subgroups \cite{Peterson1983}.}
\end{prop}

\begin{proof}
It suffices to show that the real root 1-parameter subgroups lie in the subgroup generated by the $G_\al$, since these generate $\Gg'$.  By definition a real root $\be$ is one of the form $w(\al)$ for some simple root $\al$ and $w \in W$.  Then we can write the subgroup $x_{\be}(t)$ as $\dot{w}x_{\al}(t)\dot{w}^{-1}$ for any representative $\dot{w}$ of $w$ in $\Gg'$.  But by \cref{eqn:simplerootformula} this can be written in terms of simple root 1-parameter subgroups.  
\end{proof}

\begin{remark}
We could also consider a completed version of the Kac-Moody group $\Gg$, as in \cite[6.1.16]{Kumar2002}.  In the affine case, this corresponds to using the formal loop group rather than the polynomial loop group.  However, only the smaller group $\Gg$ has a double Bruhat decomposition, since the completed group does not have a Bruhat decomposition with respect to $\Bb_-$.  Furthermore, the formal loop group does not admit evaluation representations, so it is not the right object to consider in the context of the integrable systems constructed in \cref{sec:acts}.
\end{remark}

\subsection{Affine Kac-Moody Groups}
In affine type the group $\Gg$ admits an alternative description as the central extension of a loop group.  Let $C$ be a finite type Cartan matrix, $G$ the corresponding simply connected complex algebraic group with Lie algebra $\fg$, and $\Gg$ the Kac-Moody group of the extended matrix $\wt{C}$.  If $LG := G(\CC[z^{\pm 1}])$ is the group of regular maps from $\CC^*$ to $G$, there is a universal central extension
\[
1 \longrightarrow \CC^* \longrightarrow \wt{LG} \longrightarrow LG \longrightarrow 1
\]
and an isomorphism $\Gg' \cong \wt{LG}$.  The rotation action of $\CC^*$ on $LG$ extends to $\wt{LG}$, and $\Gg$ is isomorphic with the semidirect product $\CC^* \ltimes \wt{LG}$ \cite[13.2.9]{Kumar2002}.

The central extension splits canonically over the subgroups $G(\CC[z])$ and $G(\CC[z^{-1}])$ of $LG$, so we have $\CC^* \times G(\CC[z]), \CC^* \times G(\CC[z^{-1}]) \subset \wh{LG}$.  Evaluation at $z=0$ gives a homomorphism $\CC^* \times G(\CC[z]) \to G$, and $\Bb_+$ is the preimage of the positive Borel subgroup of $G$.  Similarly $\Bb_- \subset \CC^* \times G(\CC[z^{-1}])$ is the preimage of the negative Borel subgroup of $G$ under evaluation at $z=\infty$ \cite[13.2.2]{Kumar2002}.  The Cartan subgroup $\wt{H}$ of $\wt{LG}$ splits as the product of the center of $\wt{LG}$ and the Cartan subgroup $H$ of $G$, embedded as constant maps (we write the Cartan subgroup of an affine Kac-Moody group as $\wt{H}$ to distinguish it from the Cartan subgroup of $G$).

A faithful $n$-dimensional $G$-representation yields a closed embedding $G \hookrightarrow \mathrm{Mat}_{n \times n}$, hence an inclusion $LG \hookrightarrow \mathrm{Mat}_{n \times n} \otimes \CC[z^{\pm 1}]$.  The subsets
\[
LG_m := \left\{ A(z) = \sum_{k=-m}^m A_{ij}^k z^k : A(z) \in LG \right\} \subset \mathrm{Mat}_{n \times n} \otimes \CC[z^{\pm 1}]
\]
are affine varieties, and the natural maps $LG_m \hookrightarrow LG_{m+1}$ are closed embeddings.  This defines an ind-variety structure on $LG$, which is independent of the choice of representation.  

It is clear that under this ind-variety structure the evaluation maps $LG \to G$ are regular; the same cannot be said of the ind-variety structure $LG$ inherits as a Kac-Moody group.  Our discussion of double Bruhat cells is based on the latter structure, but for integrable systems we will consider functions pulled back along evaluation maps.  Thus to ensure these yield regular functions on double Bruhat cells we must verify the compatibility of the two ind-variety structures.  This is essentially well-known, but for convenience we include a proof.  We use $LG_{pol}$ to refer to $LG$ with the ind-variety structure described in this section, and $LG_{KM}$ to refer to the ind-variety structure described in \cref{subsec:KMgroups}.

\begin{prop}
The ind-variety structures $LG_{pol}$ and $LG_{KM}$ are equivalent. That is, the identity map is a biregular isomorphism between them.  
\end{prop}

\begin{proof}
We first show that the induced structures $(\Uu_\pm)_{pol}$ and $(\Uu_\pm)_{KM}$ are equivalent (note that $\Uu_\pm$ is manifestly a closed subgroup of $LG_{pol}$).  If $w_\circ$ is the longest element of the Weyl group of $G$, $\Uu'_-(w_\circ)$ and $\Uu_-(w_\circ)$ are closed subgroups of $LG_{pol}$, and \cref{prop:unipotentfactorization} is clearly true for $(\Uu_\pm)_{pol}$.  Thus showing the claim for $\Uu_\pm$ reduces to showing it for $\Uu'_\pm(w_\circ)$. 

We now invoke the corresponding theorem about the affine Grassmannian $X:=LG/G(\CC[z])=\wt{LG}/\Pp$, where $\Pp \subset \wt{LG}$ is the parabolic subgroup corresponding to the subset $\{ \al_1, \dots ,\al_r \} \subset \{ \al_0,\dots, \al_r\}$ of simple affine roots.  Like $LG$, $X$ has two equivalent but a priori distinct ind-variety structures \cite[13.2.18]{Kumar2002}.  First, it is a disjoint union of Schubert cells $X_w = \Bb_+ \dot{w} \Pp/ \Pp$, and is filtered by finite-dimensional projective varieties
\[
X_n = \bigcup_{\ell(w) \leq n} X_w.
\]
Alternatively, $X$ can be written as an increasing union of closed subvarieties of finite-dimensional Grassmannians.  We refer the reader to \cite[13.2.15]{Kumar2002} for the precise construction, noting only that it is clear that $LG_{pol}$ acts regularly on $X$.  In particular, $\Uu'_-(w_\circ)_{pol}$ acts faithfully on the dense open subset of $\wt{LG}_0/\Pp$, and $\Uu'_-(w_\circ)_{pol} \cong \Gg_0/\Pp \cong \Uu'_-(w_\circ)_{KM}$.  The claim for $\Uu_+$ follows similarly.

In particular, the two ind-variety structures on $\Uu_- \times H \times \Uu_+$ coincide.  By \cref{prop:gaussian} this is isomorphic with an open subset $LG_0 \subset LG_{KM}$.  But it is clear that $LG_0$ is open in $LG_{pol}$, and that \cref{prop:gaussian} holds for $LG_{pol}$.  Thus the two ind-variety structures on $\wt{LG}_0$ are equivalent, and since the translates of $\wt{LG}_0$ form an open cover of $LG$ the proposition follows.  
  
\end{proof}

\begin{remark}\label{remark:singular}
All but finitely many of the varieties used in either definition of the ind-variety structure are singular, and unavoidably so: in \cite{Fishel2004} it was shown that $X$ and $LG$ cannot be written locally as an increasing union of smooth subvarieties.  Thus $LG$ is not a complex manifold, even though we have the following property: for any $g \in LG$ the canonical map
\[
\varprojlim \Sym^*(m_i(g)/m_i(g)^2) \to \varprojlim \bigoplus_{n=0}^\infty m_i(g)^n/m_i(g)^{n+1}
\]
is an isomorphism, where $m_i(g) \subset \CC[LG_i]$ is the vanishing ideal of $g$ \cite[4.3.7]{Kumar2002}.
\end{remark}

\section{Infinite-Dimensional Poisson-Lie Theory} \label{sec:plgroups}

In this section we extend several essential results of Poisson-Lie theory to the setting of ind-algebraic groups, and Kac-Moody groups in particular.  Recall that a Poisson-Lie group is a Lie group equipped with a Poisson structure such that the group operation $G \times G \to G$ is a Poisson map; we refer to \cite{Kosmann-Schwarzbach1997,Chari1995,Reyman1994} for a detailed exposition in the finite-dimensional case.  

\subsection{Standard Poisson-Lie Structure on $SL_2$}

We briefly review the standard Poisson structure on $SL_2$; this is both a model for the general case, and essential for the explicit computations we will perform in \cref{sec:brackets}.  The Lie algebra $\sl_2$ has generators
\[
X = \begin{pmatrix} 0 & 1 \\ 0 & 0 \end{pmatrix}, \quad
Y = \begin{pmatrix} 0 & 0 \\ 1 & 0 \end{pmatrix}, \quad
H = \begin{pmatrix} 1 & 0 \\ 0 & -1 \end{pmatrix},
\]
and an invariant form unique up to fixing the scalar $d := \frac{2}{(H,H)}$. If $\Om_d \in \fg \otimes \fg$ is the corresponding Casimir, we write $\Om_d = \Om_{+-} + \Om_0 + \Om_{-+}$, where $\Om_0 \in \hh \otimes \hh, \Om_{+-} \in \nn_+ \otimes \nn_-,$ and $\Om_{-+} \in \nn_- \otimes \nn_+$.  We have the standard quasitriangular $r$-matrix is
\begin{equation}\label{eqn:sl2rmatrix}
r = \Om_0 + 2\Om_{+-} = d(\frac12 H \otimes H + 2 X \otimes Y).
\end{equation}
That is, $r$ is a solution of the classical Yang-Baxter equation 
\[
[r_{12},r_{13}]+[r_{12},r_{23}]+[r_{13},r_{23}] = 0,
\]
and its symmetric part is adjoint invariant \cite[2.1.11]{Chari1995}.

Trivializing the tangent bundle by right translations, we define a Poisson bivector whose value at $g \in SL_2$ is $\Ad_g(r) - r$.  The resulting tensor is skew-symmetric since the symmetric part of $r$ is invariant, and its compatibility with the group structure is immediate by construction.  Moreover, the Yang-Baxter equation implies the Jacobi identity for the corresponding Poisson bracket \cite[4.2]{Kosmann-Schwarzbach1997}.

Given the parametrization
\[
SL_2 = \left\{ \begin{pmatrix} A & B \\ C & D \end{pmatrix} : AD - BC = 1 \right\},
\]
the Poisson brackets of the coordinate functions are
\begin{gather*}
\{B,A\} = dAB, \quad \{B,D\} = -dBD, \quad \{B,C\} = 0, \\
\{C,A\} = dAC, \quad \{C,D\} = -dCD, \quad \{D,A\} = 2dBC.
\end{gather*}
To notate the dependence of the bracket on $d$, we denote the corresponding Poisson algebraic group by $SL_2^{(d)}$.

\subsection{Poisson Ind-Varieties} \label{subsec:poissonkm}
In this section we introduce a basic formalism for infinite-dimensional Poisson algebraic geometry.  All ind-varieties are tacitly taken to be affine unless stated otherwise.

\begin{defn}
A \emph{Poisson ind-variety} is an ind-variety $X$ with a Poisson bracket on $\CC[X]$, continuous as a map $\CC[X] \otimes \CC[X] \to \CC[X]$.  A \emph{Poisson map} is a regular map of ind-varieties which intertwines the Poisson brackets on their coordinate rings.
\end{defn}

Whenever $V = \varprojlim V_i$ and $W = \varprojlim W_i$ are inverse limits of (discrete) vector spaces, we have the completed tensor product $V \wh{\otimes} W := \varprojlim V_i \otimes W_i$.  For example, if $X$ and $Y$ are ind-varieties, $\CC[X] \wh{\otimes} \CC[Y]$ is just the coordinate ring of $X \times Y$.  $V \otimes W$ sits in $V\wh{\otimes} W$ as a dense subspace with respect to its inverse limit topology, and whenever we refer to a topology on $V \otimes W$ (as in the preceding definition) we mean its subspace topology. 

\begin{remark}\label{rmk:invtop}
The role of the inverse limit topology on $V$ is to restrict our attention to operations that can be defined through the $V_i$.  A linear map $\phi: V \to W$ is continuous if and only if for each $i$ and all $k \gg 0$ there are linear maps $\phi_{ki}: V_k \to W_i$ which commute with each other, the maps defining the inverse systems, and $\phi$ in the obvious ways (note that for each $i$, $\phi_{ki}$ is defined for $k$ sufficiently large, but how large $k$ must be depends on $i$).  In other words, taking the inverse limit is a full and faithful functor from the category of pro-vector spaces indexed by $\NN$ to the category of topological vector spaces.  This allows us to go back and forth between topological statements about $V$ and purely algebraic statements about the $V_i$.  In particular, we have the following useful observation:

\begin{lemma}\label{lem:protensor}
Let $\phi: V \to A$ and $\psi: W \to B$ be continuous linear maps between inverse limits of discrete vector spaces (indexed by $\NN$).  Then $\phi \ot \psi$ extends continuously to a map $\phi \wh{\otimes} \psi: V \wh{\otimes} W \to A \wh{\ot} B$ of completed tensor products.
\end{lemma}
\begin{proof}
Since $\phi$ and $\psi$ are continuous, they are determined by collections of maps $\{ \phi_{ki}: V_k \to A_i\: |\: k \gg 0 \}$  and $\{ \psi_{ki}: W_k \to B_i\: |\: k \gg 0 \}$ as above.  But then for each $i$ we have linear maps $\phi_{ki} \otimes \psi_{ki}: V_k \otimes W_k \to A_i \otimes B_i$ for $k$ sufficiently large.  These readily satisfy the necessary compatibility requirements, hence yield a continuous linear map $\phi \wh{\ot} \psi: V \wh{\ot} W \to A \wh{\ot} B$.
\end{proof}

\end{remark}

\begin{prop}\label{prop:prods}
For any Poisson ind-varieties $X$ and $Y$, $X \times Y$ has a canonical Poisson structure.
\end{prop}
\begin{proof}
The bracket on $\CC[X] \otimes \CC[Y] \subset \CC[X \times Y]$ may be given by the usual formula $\{ f \otimes \phi, g \otimes \psi \}_{X \times Y} := \{f,g\}_X \otimes \phi \psi + fg \otimes \{\phi, \psi\}_Y$.  The fact that this extends to all of $\CC[X \times Y]$ follows from \cref{lem:protensor} and the continuity of the brackets on $X$ and $Y$.   
\end{proof}

\begin{defn}
A \emph{Poisson Ind-Group} is an ind-algebraic group $G$ which is a Poisson ind-variety and whose group operation $G \times G \to G$ is Poisson.
\end{defn}

As in the case of $SL_2$, it will be convenient to define Poisson brackets implicitly by providing a bivector field.  However, the groups we are interested in need not be inductive limits of smooth varieties (see \cref{remark:singular}), so we must be careful in discussing their tangent bundles.  The following proposition guarantees that nonetheless the trivialized tangent bundle behaves as expected.

\begin{prop}\label{prop:bivector}
Let $\Gg$ be an ind-group and $\fg$ its Lie algebra.  There is a bijection between continuous $n$-derivations of $\CC[\Gg]$ and regular maps $\Gg \to \bigwedge^n \fg$ (by $n$-derivation we mean a skew-symmetric map $\CC[\Gg] \wh{\ot} \dots \wh{\ot} \CC[\Gg] \to \CC[\Gg]$ which is a derivation in each position).  Given a map $K: \Gg \to \bigwedge^n \fg$, the corresponding $n$-derivation $\wt{K}$ takes the functions $f_1,\dots,f_n \in \CC[\Gg]$ to the function
\[
\wt{K}(f_1, \dots, f_n): g \mapsto \langle K(g) | d_e \ell_g^* f_1 \wedge \dots \wedge d_e \ell_g^* f_n \rangle.
\]
\end{prop}
\begin{proof}
We prove the case $n=1$, the higher rank case not being substantively different.  We first show that the regularity of $K$ ensures that the stated formula takes regular functions to regular functions, and that this assignment is continuous.   Note that $\fg$ is an ind-variety via its filtration by the $T_e\Gg_i$, and that there is a correspondence between regular maps $K:\Gg \to \fg$ and continuous linear maps $K^*: \fg^* \to \CC[\Gg]$.  Thus given $K$ we have a continuous linear endomorphism of $\CC[\Gg]$ given by
\[
\wt{K} := m \circ (1 \wh{\ot} K^*) \circ (1 \wh{\ot} d_e) \circ \De.
\]
Here $\De: \CC[\Gg] \to \CC[\Gg] \wh{\ot} \CC[\Gg]$ is the coproduct on $\CC[\Gg]$ and $m$ is the extension of the multiplication map to $\CC[\Gg] \wh{\ot} \CC[\Gg]$.  We have implicitly used \cref{lem:protensor} and the fact that $d_e$ is continuous.  This composition recovers the formula stated in the proposition when evaluated on a function $f \in \CC[\Gg]$, and in particular expresses it as a manifestly continuous map from $\CC[\Gg]$ to itself.  

Conversely, given a continuous derivation $\wt{K}$ of $\CC[\Gg]$, we consider the map $K^*: \CC[\Gg] \to \CC[\Gg]$ given by
\[
K^* := m \circ (S \wh{\ot} \wt{K}) \circ \De,
\]
where $S$ is the antipode of $\CC[\Gg]$.  If $m_e \subset \CC[\Gg]$ is the maximal ideal of the identity, we let the reader check that $K^*$ annihilates $m_e^2$, hence descends to a continuous linear map $K^*: \fg^* = m_e/m_e^2 \to \CC[\Gg]$.  As observed earlier, this data is equivalent to a regular map $K: \Gg \to \fg$.  Furthermore, from the defining property of the antipode it follows that this construction and the one above are inverse to each other.
\end{proof}

In particular, a Poisson structure on an ind-group $\Gg$ is determined by a Poisson bivecter $\pi: \Gg \to \bigwedge^2 \fg$.  Restating the compatibility of the bracket on $\Gg$ with the group operation in terms of $\pi$ we obtain the following definition.

\begin{defn}
A polyvector field $K: \Gg \to \bigwedge^n \fg$ is multiplicative if $K(gh) = \mathrm{Ad}_{h^{-1}}K(g) + K(h)$.
\end{defn}

\begin{remark}\label{rmk:bialgebra}
The derivative $d_e K : \fg \to \bigwedge^n \fg$ of a multiplicative polyvector field is a 1-cocycle of $\fg$ with values in $\bigwedge^n \fg$.  If $\pi$ is a Poisson bivector, then $d_e \pi$ is a Lie cobracket which makes $\fg$ a Lie bialgebra.  The dual of $d_e \pi$ is a continuous Lie bracket on $\fg^*$, which is the essentially the Poisson bracket on $\CC[\Gg]$.  That is, the maximal ideal of the identity $m_e \subset \CC[\Gg]$ is a Lie subalgebra and $m_e^2 \subset m_e$ an ideal, hence there is an induced Lie bracket on $\fg^*$.  We will not need this observation, except in \cref{sec:symp} where we describe an explicit alternative description of the bracket on $\fg^*$ in the Kac-Moody case.
\end{remark}

\subsection{Standard Poisson-Lie Structure on a Kac-Moody Group}

We now define the standard Poisson-Lie structure on a symmetrizable Kac-Moody group $\Gg$.  The construction follows the same lines as for $SL_2$ (or any semisimple Lie group), but the general case presents certain technical problems absent when considering finite-dimensional groups.\footnote{In the affine case, a different analytic approach is considered in \cite{Reshetikhin}.}

The invariant form on $\fg$ lets us identify it $\Gg$-equivariantly with a dense subspace of $\fg^*$, hence $\fg^* \wh{\otimes} \fg^*$ may be viewed as a completion of $\fg \otimes \fg$.  We denote this by $\fg \wh{\otimes} \fg$, and in particular there is an element $\Om$ of $\fg \wh{\otimes} \fg$ associated with the invariant form on $\fg$.  We write $\Om$ as $\Om_{+-} + \Om_0 + \Om_{-+}$, where $\Om_{0} \in \hh \otimes \hh$, $\Om_{+-} \in \nn_+ \wh{\otimes} \nn_-$, and $\Om_{-+} \in \nn_- \wh{\otimes} \nn_+$.  Then $r = \Om_0 + 2\Om_{+-}$ is a pseudoquasitriangular $r$-matrix \cite[Section 4]{Drinfel'd1988}; that is, $r$ satisfies the classical Yang-Baxter equation and has adjoint-invariant symmetric part, but cannot be written as a sum of finitely many simple tensors.

As in the finite-dimensional case, we want to define a Poisson bivector $\pi: \Gg \to \bigwedge^2 \fg$ by $\pi(g) = \Ad_g(r) - r$.  Now, however, $r$ is not an element of $\fg \otimes \fg$ but rather a completion thereof, so we must specifically prove that $\pi(g)$ is actually an element of $\bigwedge^2 \fg$.

\begin{thm}\label{thm:rmatrix}
The map $g \mapsto \Ad_g(r) - r$ defines a bivector field $\pi: \Gg \to \bigwedge^2 \fg$.
\end{thm}

\begin{proof}
First we check that $\Ad_g(r) - r \in \fg \ot \fg$ for all $g \in \Gg$.  We begin with the case where $g$ lies in the $SL_2$ subgroup $G_{\al}$ for some simple root $\al$.  First decompose $\fg$ as a direct sum of $G_\al$-subrepresentations corresponding to $\al$-root strings.  That is, let 
\[
\fg_{[\be]} = \bigoplus_{n \in \ZZ} \fg_{\be + n\al}, \quad \fg = \bigoplus_{[\be] \in \Qq/\ZZ\al} \fg_{[\be]},
\]
where $\Qq$ is the root lattice of $\Gg$.  Since $\al$ is simple, for any $[\be]$ we have either $\fg_{[\be]} \subset \nn_+$, $\fg_{[\be]} \subset \nn_-$, or $\be \in \ZZ\al$.  Furthermore, the invariant form on $\fg$ restricts to a nondegenerate $G_\al$-invariant pairing between $\fg_{[\be]}$ and $\fg_{[-\be]}$.  

Now we can rewrite the $r$-matrix as 
\[
r = r_\al + \sum_{\substack{[\be] \in \Qq/\ZZ\al \\ \be > 0}}r_{[\be]}.
\]
Here $r_{[\be]}$ is the element of $\fg_{[\be]} \otimes \fg_{[-\be]}$ representing their $G_\al$-invariant pairing and $r_\al \in \fg_{[\al]} \otimes \fg_{[\al]}$.  In particular, since $r_{[\be]}$ is $G_\al$-invariant, $\Ad_g(r_{[\be]}) = r_{[\be]}$ and
\[
\Ad_g(r) - r = \Ad_g(r_{[\al]}) - r_{[\al]}.
\]
The right hand side is manifestly finite-rank, hence $\Ad_g(r) - r \in \fg \ot \fg$ for $g \in G_{\al}$.

It is then straightforward to see that $\Ad_g(r) - r \in \fg \ot \fg$ whenever $g$ is a product of elements from simple root subgroups, and by \cref{prop:simplegenerators} any $g \in \Gg'$ is of this form.  Moreover, since $r$ lies in the zero weight space of $\fg \wh{\ot} \fg$ it is fixed by the Cartan subgroup $H$.  Since $\Gg$ is generated by $H$ and $\Gg'$, it follows that $\Ad_g(r) - r \in \fg \ot \fg$ for any $g \in \Gg$.  We have $\Ad_g(r) - r \in \bigwedge^2 \fg \subset \fg \ot \fg$ because the symmetric part of $r$ is adjoint invariant.  Finally, the fact that $\pi$ is regular follows from the fact that the adjoint action of $\Gg$ on $\bigwedge^2 \fg$ is regular.  
\end{proof}

By \cref{prop:bivector}, $\pi$ defines a continuous skew-symmetric bracket on $\CC[\Gg]$ satisfying the Leibniz rule.  That this bracket satisfies the Jacobi identity is a consequence of the fact that $r$ is a solution of the classical Yang-Baxter equation.  To make this precise for a general Kac-Moody group we must first introduce a certain dense subalgebra of $\CC[\Gg]$.  

Recall the embedding 
\[
\Gg \hookrightarrow V = \bigoplus_{i=1}^{\mathrm{dim}(H)} (L(\om_i) \oplus L(\om_i)^{\vee})
\]
used to define the ind-variety structure on $\Gg$.  The weight grading of $V$ expresses it as a direct sum $V = \bigoplus_{\al \in \Qq} V_{\al}$ of finite-dimensional subspaces.

\begin{defn}
The algebra of strongly regular functions on $V$ is the symmetric algebra of its graded dual, 
\[
\CC[V]_{\mathrm{s.r.}} = \Sym^*(\bigoplus_{\al \in \Qq} V^*_{\al}).
\]
The algebra $\CC[\Gg]_{\mathrm{s.r.}}$ of strongly regular functions on $\Gg$ is the image of $\CC[V]_{\mathrm{s.r.}}$ in $\CC[\Gg]$ under the restriction map.\footnote{Our use of the term ``strongly regular'' differs from that in section 2 of \cite{Kac1983}, but is consistent with Section 4 of loc. cited.}
\end{defn}

\begin{prop}\label{prop:sreg}
$\CC[\Gg]_{\mathrm{s.r.}}$ is a dense subalgebra of $\CC[\Gg]$.  For any $f \in \CC[\Gg]_{\mathrm{s.r.}}$ and $g \in \Gg$, $\ell_g^*(f)$ is again strongly regular, and the differential $d_ef$ lies in the graded dual $\fg^{\vee} := \bigoplus_{\al \in \Qq} \fg_{\al}^* \subset \fg^*$.
\end{prop}

\begin{proof}
The first and last statements are immediate.  That $\ell_g^*(f)$ is strongly regular follows from the fact that the coadjoint action of $\Gg$ on the algebraic dual $\fg^*$ preserves the graded dual of $\fg$.
\end{proof}

\begin{prop}
The bracket on $\CC[\Gg]$ defined by the bivector $\pi(g) = \Ad_g(r) - r$ satisfies the Jacobi identity.
\end{prop}
\begin{proof}
We recall the proof when $\Gg$ is a semisimple algebraic group \cite{Kosmann-Schwarzbach1997}, and then explain the necessary adjustments in the general case.  First, we write the bracket as a difference of the two brackets $\{,\}_1$ and $\{,\}_2$ defined by the bivectors $\pi_1(g) = \Ad_g(r)$ and $\pi_2(g) = r$.  Now consider separately the expressions
\[
\{ \phi, \{ \psi, \xi\}_i \}_i + \{ \psi, \{ \xi, \phi\}_i \}_i + \{ \xi, \{ \phi, \psi\}_i \}_i
\]
for $i\in\{1,2\}$ and $\phi,\psi \in \CC[\Gg]$.  On writing these out explicitly in terms of $r$ one sees that half of the terms vanish by the Yang-Baxter equation, while the remaining terms are the same for both $\{,\}_1$ and $\{,\}_2$.  Thus they cancel when we take the difference of $\{,\}_1$ and $\{,\}_2$, yielding the Jacobi identity for the original bracket. 

When $\Gg$ is infinite-dimensional, this argument fails since $\pi_1$ and $\pi_2$ are not finite-rank bivectors in the sense of \cref{prop:bivector}.  However, in light of \cref{prop:sreg}, they do define biderivations $\{,\}_1$ and $\{,\}_2$ on the algebra of strongly regular functions on $\Gg$.  Moreover, the Yang-Baxter equation implies the Jacobi identity for the bracket on $\CC[\Gg]_{\mathrm{s.r.}}$ by an identical computation as in the finite-dimensional case.  But since $\CC[\Gg]_{\mathrm{s.r.}}$ is dense in $\CC[\Gg]$ and the bracket is continuous, the proposition follows.  
\end{proof}

We call the resulting Poisson structure on $\Gg$ the standard Poisson structure.  It is essentially characterized by the following proposition.

\begin{prop}\label{prop:poissonsubgroups}
$\Gg'$ and $H$ are Poisson subgroups of $\Gg$, the latter with the trivial Poisson structure.  For any simple root $\al$, $G_{\al}$ is a Poisson subgroup isomorphic with $SL_2^{(d_\al)}$.
\end{prop}

\begin{proof}
We know that only the skew-symmetric part of $r$, which lies in $\nn_+ \wh{\otimes} \nn_- \oplus \nn_- \wh{\otimes} \nn_+ \subset \fg' \wh{\ot} \fg'$, contributes to the Poisson bivector, proving the claim for $\Gg'$.  The statement about $H$ follows from the observation that $r$ lies in the zero weight space of $\fg \wh{\ot} \fg$, hence $Ad_h(r) - r = 0$ for any $h \in H$.

In the proof of \cref{thm:rmatrix} we found that for $g \in G_{\al}$, $\pi(g) = \Ad_g(r_{[\al]}) - r_{[\al]}$, where $r_{[\al]}$ is the component of $r$ in the Lie algebra of $G_{\al}$.  But from the definition of $r$ and \cref{eqn:sl2rmatrix}, it is clear that $r_{[\al]}$ is precisely the $r$-matrix of $SL_2^{(d_\al)}$, and the proposition follows.
\end{proof}

\begin{prop}\label{prop:commuting}
\emph{(\cite[12.24]{Reyman1994})}
If $\phi, \psi \in \CC[\Gg]$ are invariant under conjugation, then 
\[
\{\phi,\psi\} = 0.
\]
\end{prop}

\begin{proof}
At any $g \in \Gg$ we check that
\begin{align*}
\{\phi, \psi\}(g) &= \langle \Ad_g(r) - r | d\phi \wedge d\psi \rangle \\
\quad &= \langle r | \Ad^*_g (d\phi \wedge d\psi) - d\phi \wedge d\psi \rangle \\
\quad &= 0,
\end{align*}
since $\Ad^*_g (d\phi \wedge d\psi) = d\phi \wedge d\psi$ by assumption.  
\end{proof}

\subsection{Double Bruhat Cells and Symplectic Leaves}\label{sec:symp}
In this section we show that the double Bruhat cells of a symmetrizable Kac-Moody group $\Gg$ are Poisson subvarieties, and in particular obtain a decomposition of $\Gg$ into symplectic leaves.  Recall that the symplectic leaves of a finite-dimensional Poisson manifold are the orbits of its piecewise Hamiltonian flows, have canonical symplectic structures, and define a generalized foliation of $\Gg$.  The existence of symplectic leaves in $\Gg$ is nontrivial, since a vector field on a general ind-variety need not have integral curves even if the ind-variety is smooth.  

We will obtain an explicit characterization of the symplectic leaves of $\Gg$ in \cref{thm:leaves}, but first we offer an elementary proof of their existence.  We will use \cref{prop:fdimdbc,prop:pfact} from \cref{sec:dbc}, but their proofs do not rely on the results of this section.

\begin{prop}\label{prop:poisdbc}
The double Bruhat cells $\Gg^{u,v}$ are Poisson subvarieties of $\Gg$.  
\end{prop}
\begin{proof}
In \cref{prop:pfact} we construct dominant Poisson map $\phi_{\mb{i}}$ from a Poisson variety to $\Gg^{u,v}$.  It follows that the closure of $\Gg^{u,v}$ in $\Gg$ is a Poisson subvariety: the kernel of $\phi_{\mb{i}}^*$ in $\CC[\Gg]$ is an open Poisson ideal, hence the closure of $\Gg^{u,v}$ is the (maximal) spectrum of the Poisson algebra $\CC[\Gg]/\mathrm{ker}\phi_{\mb{i}}^*$.  The closure of $\Gg^{u,v}$ is can be explicitly written as
\[
\bigcup_{u'\leq u, v' \leq v} \Gg^{u',v'},
\]
and in particular $\Gg^{u,v}$ is the complement of a divisor in its closure.  But such an open subset of an affine Poisson variety inherits a canonical Poisson structure \cite[2.35]{Vanhaecke2001}.
\end{proof}

\begin{corollary}
The group $\Gg$ is the disjoint union of finite-dimensional symplectic leaves.  
\end{corollary}
\begin{proof}
Follows from \cref{prop:poisdbc} and the fact that double Bruhat cells are smooth and finite-dimensional (\cref{prop:fdimdbc}).
\end{proof}

We can get a more precise description of the symplectic leaves of $\Gg$ by introducing the dual group $\Gg^{\vee}$ and the double group $\Dd$.  These are ind-groups defined by
\[
\Gg^{\vee} := \{(b_-,b_+) \in \Bb_- \times \Bb_+\: |\: [b_-]_0 = [b_+]_0^{-1}\}, \quad \Dd := \Gg \times \Gg.
\]
The dual group $\Gg^{\vee}$ sits inside $\Dd$ in the obvious way, and we view $\Gg$ as a subgroup of $\Dd$ via its diagonal embedding.  

\begin{thm}\label{thm:leaves}
The symplectic leaves of a symmetrizable Kac-Moody group $\Gg$ are the connected components of its intersections with the double cosets of $\Gg^{\vee}$ in $\Dd$.
\end{thm}

The proof of this theorem proceeds in several steps, closely following \cite{Lu1990} in the finite-dimensional case.  The idea of the proof remains the same, but we indicate how some arguments must be rephrased or altered to remain valid in the current setting.  In particular, one does not expect a priori to have such a theorem for arbitrary Poisson ind-groups, as at several points we must appeal to particular properties of Kac-Moody groups and their standard Poisson structure.

First note that the Lie algebra of $\Gg^{\vee}$ is 
\[
\fg^{\vee} = \{ (X_-,X_+) \in \bb_- \oplus \bb_+\: |\: [X_-]_0 = - [X_+]_0 \},
\]
where $[X_{\pm}]_0$ denotes the component of $X_{\pm}$ in $\hh$.  The Lie algebra $\dd = \fg \oplus \fg$ of $\Dd$ is then the direct sum of $\fg^{\vee}$ and $\fg$, the latter embedded diagonally.  Moreover, $\fg^{\vee}$ and $\fg$ are maximal isotropic subalgebras under the nondegenerate invariant form
\[
\langle(X_1,Y_1),(X_2,Y_2)\rangle = \langle X_1,X_2 \rangle - \langle Y_1,Y_2 \rangle.
\]
In particular, this form identifies $\fg^{\vee}$ with the graded dual of $\fg$, justifying its notation.\footnote{Though one can intrinsically define the Lie algebra structure on $\fg^*$ for an arbitrary Poisson ind-group (\cref{rmk:bialgebra}), one cannot expect the existence of a corresponding dual group in general, since Lie's third theorem fails in this generality.}

Given this identification, the bracket on $\dd$ can be rewritten in terms of the coadjoint actions of $\fg$ and $\fg^{\vee}$ on each other.  That is, if $X_1,X_2 \in \fg$ and $Y_1,Y_2 \in \fg^{\vee}$, then
\begin{gather}\label{eqn:double}
[(X_1,Y_1),(X_2,Y_2)] = ([X_1,X_2] + ad^*_{Y_1}X_2 - ad^*_{Y_2}X_1, [Y_1,Y_2] + ad^*_{X_1}Y_2 - ad^*_{X_2}Y_1).
\end{gather}

\begin{defn}
Let $\pi$ be the standard Poisson bivector on $\Gg$.  For any $\mu \in \fg^*$ we define the (left) dressing vector field as
\[
X_{\mu} := \io_{\mu}(\pi).
\]
\end{defn}
Taken together these yield a continuous map $X: \fg^* \wh{\ot} \CC[\Gg] \to \CC[\Gg]$ which is a derivation in the right component.  Furthermore, one can recover the Poisson bivector $\pi$ from $X$.  Explicitly, the map
\[
m \circ (X_{13} \wh{\ot} S_2) \circ (1 \wh{\ot} \De): \fg^* \wh{\ot} \CC[\Gg] \to \CC[\Gg]
\]
factors through $\fg^* \wh{\ot} \fg^*$ as in the proof of \cref{prop:bivector}, and is dual to the map $\pi: \Gg \to \bigwedge^2 \fg$.  Here $\De$ is the coproduct on $\CC[\Gg]$, $S$ is the antipode, $m$ is multiplication in $\CC[\Gg]$, and the notation $X_{13}$ means we apply $X$ to the first and third terms of $\fg^* \wh{\ot} \CC[\Gg] \wh{\ot} \CC[\Gg]$.

\begin{lemma}\label{lem:multfields}
Let $K$ be a multiplicative polyvector field.  (1) If $X$ is a left-invariant vector field, $\Ll_X K$ is also left-invariant.  Here $\Ll_X K$ is the Lie derivative of $K$ with respect to $X$. (2) If $d_e(K) = 0$, then $K$ is identically zero.
\end{lemma}
\begin{proof}
We take $K$ to be a vector field, the higher rank case being similar.  

(1) Left-invariance of $X$ is equivalent to $\De \circ X = (1 \wh{\ot} X) \circ \De$, and multiplicativity of $K$ is equivalent to $\De \circ K = (1 \wh{\ot} K) \circ \De + (K \wh{\ot} 1) \circ \De$.  Then $\Ll_X K$ is left-invariant by the following equality of maps from $\CC[\Gg]$ to $\CC[\Gg] \wh{\ot} \CC[\Gg]$:
\begin{align*}
\De \circ \Ll_XK & = \De \circ (X \circ K - K \circ X)\\ 
&= (1 \wh{\ot} X) \circ (K \wh{\ot} 1 + 1 \wh{\ot} K) \circ \De - (K \wh{\ot} 1 + 1 \wh{\ot} K) \circ (1 \wh{\ot} X) \circ \De\\
&= (1 \wt{\ot} \Ll_XK) \circ \De
\end{align*}

(2) Since $d_e(K) = 0$, $\Ll_XK|_e = 0$ for any left-invariant $X$.  But $\Ll_XK$ is itself left-invariant by (1), hence is identically zero.  In particular, since we can integrate the left-invariant vector fields corresponding to the real root spaces, $K$ is invariant under left translation by the corresponding 1-parameter subgroups.  Since $\Gg$ is generated by these subgroups and $H = \mathrm{exp}(\hh)$, $K$ is invariant under all left-translations.  But $K$ is multiplicative, hence $K|_e = 0$ and $K$ must then be identically zero.
\end{proof}

\begin{prop}\label{prop:dressing}
The dressing fields $X_{\mu}$ satisfy the twisted multiplicativity condition
\[
X_{\mu}(gh) = X_{\mu}(h) + Ad_{h^{-1}}[ X_{Ad_{h^{-1}(\mu)}}(g)],
\]
and the derivative $d_e X_{\mu}: \fg \to \fg$ is the coadjoint action $\mathrm{ad}_{\mu}^*$.
Moreover, $X: \fg^* \wh{\ot} \CC[\Gg] \to \CC[\Gg]$ is the only continuous derivation satisfying these properties.
\end{prop}
\begin{proof}
Twisted multiplicativity of the dressing fields follows readily from the definition of multiplicativity.  Likewise, the fact that $X_{\mu} = \mathrm{ad}_{\mu}^*$ follows from unwinding the definition of the bracket on $\fg^*$. We omit the calculations, which resemble those of \cref{prop:bivector,lem:multfields}.  

Suppose $Y: \fg^* \wh{\ot} \CC[\Gg] \to \CC[\Gg]$ is a continuous derivation and satisfies the given properties.  In the same way that we can recover $\pi$ from $X$, we recover a bivector field $\wt{Y}$ from $Y$.  The twisted multiplicativity of $Y$ is again equivalent to the multiplicativity of $\wt{Y}$, and $d_e \wt{Y} = d_e \pi$ since the derivatives of $X$ and $Y$ coincide at the identity.  The difference $\pi - \wt{Y}$ is then multiplicative bivector field whose derivative at the identity is zero.  Then by \cref{lem:multfields} $\pi - \wt{Y}$ is identically zero, hence $X = Y$.
\end{proof}

Consider the left action of $\Gg^{\vee}$ on $\Dd/\Gg^{\vee}$, and the induced action of $\fg^{\vee}$ by vector fields.  Note that the quotient of $\Dd/\Gg^{\vee}$ exists as an ind-variety; $\Dd/(\Bb_-\times \Bb_+)$ is a product of opposite affine Grassmannians, and $\Dd/\Gg^{\vee}$ is a torus bundle over it (compare with \cite[7.2]{Kumar2002}).  The fibers of the projection from $\Gg$ to $\Dd/\Gg^{\vee}$ are the orbits of right multiplication by $\Ga := \Gg \cap \Gg^{\vee}$.  This intersection is a finite group, specifically the group of square roots of the identity in $H$.  The image of $\Gg$ in $\Dd/\Gg^{\vee}$ is open by the following proposition and the fact that the quotient map $\Gg \to \Gg/\Bb_{\pm}$ is open \cite[7.4.10]{Kumar2002}.  

\begin{prop}\label{prop:open}
The image of the multiplication map $\Gg \times \Gg^{\vee} \to \Dd$, which is the same as the image of $\Gg \times (\Bb_- \times \Bb_+) \to \Dd$, is the open set $\{(g,g')\:|\: g^{-1}g' \in \Gg_0 \}$.  Here $\Gg_0$ is the image of $\Uu_- \times H \times \Uu_+$ in $\Gg$ as in \cref{prop:gaussian}.  Similarly, the image of $\Gg^{\vee} \times \Gg \to \Dd$ is the open set $\{(g,g')\:|\: g(g')^{-1} \in \Gg_0 \}$.
\end{prop}
\begin{proof}
If $(g,g')=(kb_-,kb_+)$ for some $k \in \Gg$, $(b_-,b_+) \in \Gg^{\vee}$, then $g^{-1}g' = b_-^{-1}b_+ \in \Gg_0$.  Conversely, if $g^{-1}g' \in \Gg_0$ choose $u_{\pm} \in \Uu_{\pm}$ and $h \in H$ such that $g^{-1}g' = u_-h^2u_+$.  Then in $\Dd$ we have the factorization
\[
(g,g')=(gu_-h,gu_-h)\cdot(h^{-1}u_-^{-1},hu_+),
\]
proving the first claim.  The second then follows by taking the inverses of the two subsets considered in the first statement. 
\end{proof}

In particular the map $\Gg \to \Dd/\Gg^{\vee}$ induces isomorphisms on the tangent spaces at every point.  Thus we can pull back vector fields on $\Dd/\Gg^{\vee}$ to vector fields on $\Gg$. 

\begin{prop}\label{prop:action}
Pulling back the vector fields on $\Dd/\Gg^{\vee}$ corresponding to the infinitesimal left action of $\fg^{\vee}$, we obtain exactly the dressing vector fields on $\Gg$.
\end{prop}
\begin{proof}
We apply the uniqueness statement of \cref{prop:dressing}.  That these vector fields linearize to the coadjoint action at the identity follows from \cref{eqn:double}.  Twisted multiplicativity follows from differentiating the following version at the group level.  

Consider the open set $\Dd_0 = \{(g,g')\: |\: g^{-1}g', g(g')^{-1} \in \Gg_0 \}$.  By \cref{prop:open}, any element of $\Dd_0$ can be written as $d \cdot g$ for some $(d,g) \in \Gg^{\vee} \times \Gg$.  We can also factor it as $g^d \cdot d^g$ for some $(g^d,d^g) \in \Gg \times \Gg^{\vee}$, where $g^d$ and $d^g$ are uniquely defined up to right and left multiplication by $\Ga$, respectively.  In particular, the (local) left action of $\Gg^{\vee}$ on the image of $\Gg$ in $\Dd/\Gg^{\vee}$ can be written $\ell_d: g\Gg^{\vee} \mapsto g^d\Gg^{\vee}$.  But now by considering an element of the form $ghd$, where $g,h \in \Gg$, we obtain the identity $(g \cdot h)^d = g^d \cdot h^{(d^{g})}$.  This equality must be taken modulo the action of $\Ga$.  However, since $\Ga$ is finite it is strictly true in a neighborhood of $e \in \Gg^{\vee}$ in the analytic topology, and this is sufficient to obtain the corresponding statement about the infinitesimal action of $\fg^{\vee}$ as in \cite{Lu1990}.
\end{proof}

\begin{proof}[Proof of \cref{thm:leaves}]
The orbits of the action of $\Bb_{\pm}$ on $\Gg/\Bb_{\pm}$ are Schubert cells, which in particular are smooth finite-dimensional subvarieties.  It follows straightforwardly that the orbits of the action of $\Gg^{\vee}$ on $\Dd/\Gg^{\vee}$ are also smooth finite-dimensional subvarieties, and since $\Gg \to \Dd/\Gg^{\vee}$ is \'{e}tale the same is true of the preimages of these orbits in $\Gg$. 

 By \cref{prop:action}, the tangent space to such a preimage at any $g \in \Gg$ is exactly the span of the dressing vector fields at that point.  Note that the span of the $X_{\mu}|_g$ in $T_g \Gg$ for $\mu \in \fg^{\vee}$ is the same as the span of the $X_{\mu}|_g$ with $\mu$ arbitrary, since this subspace is finite-dimensional and $\fg^{\vee}$ is dense in $\fg^*$.  Thus the connected components of the preimages of the $\Gg^{\vee}$-orbits in $\Dd/\Gg^{\vee}$ are symplectic leaves of $\Gg$.  But these are exactly the intersections of $\Gg$ with the double cosets of $\Gg^{\vee}$ in $\Dd$.
\end{proof}

The intersections of $\Gg$ with the double cosets of $\Gg^{\vee}$ are characterized by the following theorem.  This was proved in the finite-dimensional case in \cite{Kogan2002} and \cite{Hoffmann2000}, and with \cref{thm:leaves}; the proofs given there apply verbatim in the general case.  

\begin{thm}
Given $u,v \in W$, let $H^{u,v} \subset H$ be the subgroup of elements of the form $(\dot{u}^{-1}h^{-1}\dot{u})(\dot{v}^{-1}h\dot{v})$, and let $S^{u,v} = \{g \in G^{u,v} \: | \: [\dot{u}^{-1}]_0\dot{v}^{-1}[g\dot{v}^{-1}]_0\dot{v} \in H^{u,v} \}$.  Then the intersections of $\Gg^{u,v}$ with the double cosets of $\Gg^{\vee}$ in $\Dd$ are the subsets $S^{u,v}\cdot h$ for $h \in H$.  
In particular, the symplectic leaves of a fixed double Bruhat cell are isomorphic with one another.
\end{thm}

\section{Double Bruhat Cells in Kac-Moody Groups} \label{sec:dbc}

In this section we establish the main properties of double Bruhat cells in Kac-Moody groups that will be needed in the construction of integrable systems in \cref{sec:acts}.  In particular, we generalize the factorization coordinates of \cite{Fomin1999} to the Kac-Moody setting, and describe the standard Poisson bracket in these coordinates.  We perform these computations explicitly for Coxeter double Bruhat cells in affine type, which is the relevant case for the integrable systems we consider.  
\subsection{Factorization in Double Bruhat Cells} \label{sec:dbcbasics}

Let $\Gg$ be a symmetrizable Kac-Moody group and $\Gg^{u,v}$ a fixed double Bruhat cell.  Following \cite{Fomin1999} in the case where $\Gg$ is a semisimple Lie group, we now prove that $\Gg^{u,v}$ is a rational variety, calculate its dimension, and show that on certain dense open sets it may be factored as a product of 1-parameter subgroups.  

\begin{prop}\label{prop:fdimdbc}
The image of the diagonal map
\[
\Gg^{u,v} \to \Bb_+ \dot{u} \Bb_+ / \Bb_+ \times \Bb_- \dot{v} \Bb_- / \Bb_- \cong \Uu_+(u) \times \Uu_-(v)
\]
is the open subset $U = \{(y_+,y_-) \in \Uu_+(u) \times \Uu_-(v)\: |\: \dot{v}^{-1} y^{-1}_- y_+ \dot{u} \in \Gg_0\}$.  There is an isomorphism $\Th: U \times H \to \Gg^{u,v}$ given by 
\[
\Th(y_+,y_-,h) = y_+\dot{u}[\dot{v}^{-1}y_-^{-1}y_+\dot{u}]_+^{-1}h, \quad \Th^{-1}(g) = ([\dot{u}^{-1}g]_-,[g^{-1}\dot{v}]_+^{-1},[\dot{u}^{-1}g]_0).
\]
In particular, $\Gg^{u,v}$ is a finite-dimensional variety isomorphic with a Zariski open subset of $\CC^{m+\mathrm{dim}(H)}$, where $m = \ell(u)+\ell(v)$.
\end{prop}
\begin{proof}
By \cref{prop:open} the image of $\Gg$ in $\Gg/\Bb_- \times \Gg/\Bb_+$ is $\{(g\Bb_-,g'\Bb_+)\:|\: g^{-1}g' \in \Gg_0 \}$.  The first claim follows from restricting this statement to $\Gg^{u,v}$ and using \cref{prop:refinedBruhat} to write an element of $\Gg^{u,v}$ as $y_+\dot{u}b_+ = y_- \dot{v}b_-$ for some $y_+ \in \Uu_+(u), y_- \in \Uu_-(v)$, and $b_{\pm} \in \Bb_{\pm}$.  

Verifying that the stated maps are inverse to each other is an elementary calculation, which we omit.
\end{proof}

\begin{defn}
(\cite{Fomin1999}) A \emph{double reduced word} $\mb{i}=(i_1, \dots, i_m)$ for $(u,v)$ is a shuffle of a reduced word for $u$ written in the alphabet $\{-1,\dots,-r\}$ and a reduced word for $v$ written in the alphabet $\{1,\dots,r\}$.  For each $\mb{i}$ we have a map
\[
x_{\mb{i}}: H \times (\CC^*)^{\ell(u) + \ell(v)} \to \Gg, \quad (a, t_i, \dots , t_m) \mapsto ax_{i_1}(t_1) \cdots x_{i_m}(t_m).
\]
Here $x_i(t)$ and $x_{-i}(t)$ denote the 1-parameter subgroups corresponding to $\al_i$ and $-\al_i$, respectively.  The following proposition demonstrates that $x_{\mb{i}}$ yields explicit coordinates on a dense subset of $\Gg^{u,v}$; we refer to these as factorization coordinates.
\end{defn}

\begin{prop}\label{prop:image}
The map $x_{\mb{i}}$ is injective and its image $\Gg_{\mb{i}}$ is a dense subset of $\Gg^{u,v}$.
\end{prop}

\begin{proof}
First we show that the image of $x_{\mb{i}}$ is contained in $G^{u,v}$.  For each $1 \leq i \leq r$, we have $x_i(t) \in \Bb_+$ and $x_{-i}(t) \in \Bb_+ s_i \Bb_+$.  Thus if ${k_1 < \dots < k_{\ell(u)}} \subset \{ 1,\dots,m \}$ are the indices of the negative entries in $\mb{i}$, 
\[
x_{\mb{i}}(a, t_i, \dots , t_m) \in \Bb_+ \cdots \Bb_+ s_{|i_{k_1}|} \Bb_+ \cdots \Bb_+ s_{|i_{k_{\ell(u)}}|} \Bb_+ \cdots \Bb_+.
\]
Recall that for $w, w' \in W$, 
\[
\Bb_+ w \Bb_+ \cdot \Bb_+ w' \Bb_+ = \Bb_+ w w' \Bb_+,
\]
whenever $\ell(ww') = \ell(w) + \ell(w')$ \cite[5.1.3]{Kumar2002}.  Thus in particular $x_{\mb{i}}(a, t_i, \dots , t_m) \in \Bb_+ u \Bb_+$, and by the same argument $x_{\mb{i}}(a, t_i, \dots , t_m) \in \Bb_- v \Bb_-$.

Suppose that
\[
x_{\mb{i}}(a, t_1, \dots , t_m) = x_{\mb{i}}(a', t'_1, \dots , t'_m)
\] but $(a, t_1, \dots , t_m) \neq (a', t'_1, \dots , t'_m)$, and let $k$ be the largest index such that $t_k \neq t'_k$.  Note that $\mb{i}' = (i_1, \dots ,i_k)$ is a double reduced word for some $(u',v')$, and that $x_{\mb{i}'}(a, t_i, \dots , t_k) = x_{\mb{i}'}(a', t'_i, \dots , t'_k)$. 

Multiplying both sides on the right by $x_{i_k}(-t'_k)$, we obtain
\[
x_{\mb{i}'}(a, t_i, \dots , t_k - t'_k) = x_{\mb{i}''}(a', t'_i, \dots , t'_{k-1}),
\]
where $\mb{i}''=(i_1, \dots, i_{k-1})$.  But by the first part of the proposition the left and right sides are in different double Bruhat cells, hence by contradiction $x_{\mb{i}}$ must be injective.  

Furthermore, since we know $\Gg^{u,v}$ is an open subvariety of $\CC^{m + r}$ and $x_{\mb{i}}$ is a regular map from $(\CC^*)^{m + r}$, the fact that the image of $x_{\mb{i}}$ is dense follows from its injectivity.
\end{proof}

\begin{remark}
The analogous statements for the derived subgroup $\Gg'$ follow along the same lines, with $H$ replaced by $H' = H \cap \Gg'$.  
\end{remark}

\subsection{Poisson Brackets on Double Bruhat Cells}

Recall from \cref{sec:symp} that the double Bruhat cell $\Gg^{u,v}$ is a Poisson subvariety of $\Gg$.  By modifying the map $x_{\mb{i}}$ of the previous section, we now realize the symplectic leaves of $\Gg^{u,v}$ (more precisely, their intersections with $\Gg_i$) as reductions of a Hamiltonian torus action.  In particular, we obtain modified factorization coordinates along with explicit formulas for their Poisson brackets.

First observe that $SL_2^{(d)}$ has two distinguished symplectic leaves
\begin{equation*}
S^d_+ = \left\{ \begin{pmatrix} A & B \\ 0 & A^{-1} \end{pmatrix}: A,B \neq 0 \right\}, \quad
S^d_- = \left\{ \begin{pmatrix} D^{-1} & 0 \\ C & D \end{pmatrix}: C,D \neq 0 \right\}.
\end{equation*}
The Poisson brackets on $S^d_+$ and $S^d_-$ are given by $\{B,A\} = dAB$ and $\{D,C\} = dCD$, respectively.  Now define a symplectic variety
\[
S_{\mb{i}} := S_{\ep(i_1)}^{|d_{i_1}|} \times \cdots \times S_{\ep(i_m)}^{|d_{i_m}|},
\]
where $\ep(i_j)$ is the sign of $i_j$.

If $H_k$ is the Cartan subgroup of $G_{\al_k}$, we also define two tori
\[
H_{\mb{i}} := (H/H') \times \prod_{n_{\mb{i}}(k) = 0} H_k, \quad \wh{H}_{\mb{i}} := \prod_{n_{\mb{i}}(k) \neq 0} H_k^{n_{\mb{i}}(k) - 1}.
\]
Here $n_{\mb{i}}(k)$ is the total number of times the simple reflection $s_k$ appears in our reduced expressions for $u$ and $v$, that is,
\[
n_{\mb{i}}(k) = \#\{j : |i_j| = k, 1 \leq j \leq m\}.
\]
As before, $H' = H \cap \Gg'$ is the subgroup of $H$ generated by the coroots.  

\begin{defn}\label{def:modfact}
 Let $\phi_{\mb{i}}$ be the map given by
\[
\phi_{\mb{i}}: H_{\mb{i}} \times S_{\mb{i}} \to \Gg^{u,v}, \quad (a, g_{i_1},\dots ,g_{i_m}) \mapsto a \cdot \phi_{i_1}(g_{i_1}) \cdots \phi_{i_m}(g_{i_m}).
\]
We can define a similar map for the derived subgroup $\Gg'$ by omitting the $H/H'$ factor in the definition of $H_{\mb{i}}$.
\end{defn}

\begin{prop}\label{prop:pfact}
The map $\phi_{\mb{i}}$ is Poisson, with $H_{\mb{i}}$ being given the trivial Poisson structure.  Its image is $\Gg_{\mb{i}}$ and its fibers are the orbits of a simply transitive action of $\wh{H}_{\mb{i}}$. 
\end{prop}

\begin{proof}
The first assertion follows from \cref{prop:poissonsubgroups}.  That the image of $\phi_{\mb{i}}$ is $\Gg_{\mb{i}}$ follows from a straightforward comparison of the definitions of $\phi_{\mb{i}}$ and $x_{\mb{i}}$.  We describe the action of $\wh{H}_{\mb{i}}$ by considering each of the $H_k^{n_{\mb{i}}(k) - 1}$ factors individually.  For each $k$ let $j_1 < \dots < j_{n_{\mb{i}}(k)}$ be the indices such that $|i_{j_n}| = k$.  Then for any element $t_n^{h_k}$ of the $n$th $H_k$ factor, where $1 \leq n \leq n(k) - 1$, let
\begin{align*}
t_n^{h_k} \cdot (a, g_{i_1},\dots ,g_{i_m}) & = (a, g_{i_1}, \dots, g_{i_{j_n}} \cdot t_n^{h_k}, \dots, t_n^{-h_k} \cdot g_{i_{\ell}} \cdot t_n^{h_k}, \dots\\
& \quad \dots, t_n^{-h_k} \cdot g_{i_{j_{n+1}}}, \dots, g_{i_m}).
\end{align*}
Here $t_n^{h_k} \cdot g_{i_{\ell}} \cdot t_n^{-h_k}$ refers to the conjugation action of $\phi_k(H_k)$ on $\phi_{i_{\ell}}(S_{\pm})$.  
\end{proof}

In particular, $\phi_{\mb{i}}$ induces an isomorphism between the invariant ring $\CC[H_{\mb{i}} \times S_{\mb{i}}]^{\wh{H}_{\mb{i}}}$ and the coordinate ring $\CC[\Gg_{\mb{i}}]$.  Since we know the Poisson brackets of the coordinate functions on $H_{\mb{i}} \times S_{\mb{i}}$, we obtain an explicit description of the Poisson structure of $\Gg_{\mb{i}}$.

\subsection{Affine Coxeter Double Bruhat Cells}\label{sec:brackets}

We now specialize the preceding discussion to the affine case $\Gg' \cong \wt{LG}$, and explicitly calculate the factorization coordinates and their Poisson brackets for a distinguished class of double Bruhat cells.  We will also consider the quotient of $\wt{LG}^{u,v}$ by the conjugation action of $H$.

\begin{defn}
If $u$ and $v$ are Coxeter elements of the affine Weyl group we say that $\wt{LG}^{u,v}$ is a \emph{Coxeter double Bruhat cell}.  Recall that $w \in W$ is a Coxeter element if in some (hence any) reduced expression for $w$ each simple reflection appears exactly once. 
\end{defn}

We may write any reduced word for $v$ as $s_{\si(0)} \dots s_{\si(r)}$ for some permutation $\si \in S_{r+1}$, and likewise any reduced word for $u$ as $s_{\tau(0)}\dots s_{\tau(r)}$ for some permutation $\tau$.  Given reduced words for $u$ and $v$, we will only explicitly write out the factorization coordinates for the unshuffled double reduced word $\mb{i}=(s_{\si(0)}\dots s_{\si(r)}s_{\tau(0)} \dots s_{\tau(r)})$.  This will simplify our notation but still let us perform the calculations needed in \cref{sec:acts}.

The map $\phi_{\mb{i}}$ of \cref{def:modfact} now takes the form
\begin{multline*}
\phi_{\mb{i}}: (g_{\si(0)},\dots ,g_{\si(r)},g_{\tau(0)}',\dots,g'_{\tau(r)}) \mapsto \\
\phi_{\si(0)}(g_{\si(0)}) \dots \phi_{\si(r)}(g_{\si(r)}) \phi_{\tau(0)}(g'_{\tau(0)}) \dots \phi_{\tau(r)}(g_{\tau(r)}'),
\end{multline*}
where
\[
(g_{\si(0)},\dots ,g_{\si(r)},g_{\tau(0)}',\dots,g'_{\tau(r)}) \in S_{\mb{i}} = S_+^{d_{\si(0)}} \times \cdots \times S_+^{d_{\si(r)}} \times S_-^{d_{\tau(0)}} \times \cdots \times S_-^{d_{\tau(r)}}.
\]
We will let $A_i, B_i$ and $C_i, D_i$ denote the standard coordinates on $S_+^{d_i}$ and $S_-^{d_i}$, respectively.

Since $u$ and $v$ are Coxeter elements, the torus $\wh{H}_{\mb{i}}$ is equal to $\prod_{k=0}^r H_k$, and its action on $S_{\mb{i}}$ is given by
\begin{align*}
t^{h_k} \cdot (g_{\si(0)}, \dots, g'_{\tau(r)}) & = (g_{\si(0)},\dots ,g_k\cdot t^{h_k}, \dots, t^{-h_k} \cdot g_{\si(r)} \cdot t^{h_k}, t^{-h_k} \cdot g'_{\tau(0)} \cdot t^{h_k}, \dots\\
& \quad \dots, t^{-h_k} \cdot g'_k, \dots, g'_{\tau(r)}).
\end{align*}
To write this in coordinates we introduce the notation $i <_\si k$ to mean $\si^{-1}(i) < \si^{-1}(k)$, or simply that $i$ appears to the left of $k$ in the reduced word for $v$; likewise we define $i <_\tau k$.  Then we have
\[ t^{h_k}: (A_i,B_i) \to 
\begin{cases} 
(A_i,B_i) & i <_\si k \\
(tA_i,t^{-1}B_i) & i = k \\
(A_i,t^{-C_{ki}}B_i) & i >_\si k 
\end{cases},\quad  (C_i,D_i) \to 
\begin{cases} 
(C_i,t^{C_{ki}}D_i) & i <_\tau k \\
(tC_i,tD_i) & i = k \\
(C_i,D_i) & i >_\tau k
\end{cases},
\]
where $C_{ki}$ is the corresponding entry in the Cartan matrix of $\wt{LG}$.  If we let
\[ 
T_i=A_iD_i^{-1}, \quad V_i=B_iD_i(\prod_{k <_\si i} D_k^{C_{ki}}), \quad W_i=(\prod_{k >_\tau i}A_k^{-C_{ki}})A_i^{-1}C_i, 
\]
then
\[ 
\CC[\widetilde{LG}_\mb{i}] \cong \CC[S_\mb{i}]^{\wh{H}_{\mb{i}}} \cong \CC[T^{\pm 1}_0,V^{\pm 1}_0,W^{\pm 1}_0,\dots ,T^{\pm 1}_r,V^{\pm 1}_r,W^{\pm 1}_r].
\]

In \cref{sec:acts} we will consider the quotient of $\wt{LG}^{u,v}$ by the adjoint action of $H$.  This is again a Poisson variety, since $H$ acts by Poisson automorphisms.   This is similar to the reduced double Bruhat cells considered in \cite{Zelevinsky2000,Yang2008}, though they consider the quotient by left multiplication rather than conjugation.  We now derive coordinates on $\wt{LG}^{u,v}/H$ along with their Poisson brackets. 

If $h^k \in \hh$ satisfies $\al_i(h^k)=\de_{i,k}$, then for $k \neq 0$ we have
\[ t^{h^k}: (T_i,V_i,W_i) \to 
\begin{cases} 
(T_i,t^{-\th_k}V_i,t^{\th_k}W_i) & i = 0 \\
(T_i,tV_i,t^{-1}W_i) & i = k \\
(T_i,V_i,W_i) & i \neq 0, k.
\end{cases}
\]
Now setting $S_i=V_iW_i$ and $Q=V_0(\prod_{i\neq0}V_i^{\th_i})$, a straightforward calculation yields
\begin{gather}\label{eq:coords}
\CC[\widetilde{LG}_\mb{i}/H] \cong \CC[T^{\pm 1}_0,S^{\pm 1}_0,\dots,T^{\pm 1}_r,S^{\pm 1}_r,Q^{\pm 1}].  
\end{gather}
The Poisson structure is determined by the pairwise brackets of these generators; the nonzero ones are exactly
\begin{gather}
\{S_i,T_k\}=2d_iS_iT_i\de_{i,k}, \quad \{Q, T_k\}= d_k\th_kQT_k, \notag\\
\{S_i,S_k\}=2d_kC_{ki}([i >_\si k >_\tau i] - [i >_\tau k >_\si i]) S_i S_k, \label{eq:explicitbrackets} \\
\quad \{Q,S_k\} = \biggl(\sum_{i \neq k} \th_id_kC_{ki}([i >_\si k >_\tau i] - [i >_\tau k >_\si i])\biggr)Q S_k. \notag
\end{gather}
Here $[i >_\si k >_\tau i]$ is equal to 1 if both $i >_\si k$ and $k >_\tau i$, and is equal to 0 otherwise (also recall that $\th_0 = 1$ by convention).  

In particular, though the dimensions of the symplectic leaves of $\wt{LG}^{u,v}$ depend on the specific choice of $u$ and $v$, our computations of the bracket on $\wt{LG}_{\mb{i}}/H$ imply the following:

\begin{prop}\label{prop:leafdim}
The symplectic leaves of $\widetilde{LG}_\mb{i}/H$ are of dimension $2r+2$, and $Q^2(\prod_k S_k^{-\th_k})$ is a Casimir.
\end{prop}

\section{Affine Coxeter-Toda Systems} \label{sec:acts}

We now apply the results of the preceding sections to construct integrable systems on the reduced Coxeter double Bruhat cells of $\wt{LG}$.   We will use the factorization coordinates described in \cref{sec:brackets} to work explicitly with their Hamiltonians, and to identify the relativistic periodic Toda system as an example of the construction.  

\subsection{Complete Integrability}

We first recall the following definition:

\begin{defn}\label{defn:intsys}
A completely integrable Hamiltonian system on an affine Poisson variety is a collection of Poisson-commuting functions $H_1, \dots, H_n$ whose associated Hamiltonian vector fields are generically independent, and whose number is half the dimension of a generic symplectic leaf (this is the maximum possible number given the independence requirement).
\end{defn}

Invariant functions on $\wt{LG}$ Poisson commute with each other by \cref{prop:commuting}, and we will construct such functions as follows.  Any regular function on $G$ can be pulled back along the evaluation map $\wt{LG} \times \CC^* \to G$ to a regular function on $\wt{LG} \times \CC^*$.  Choosing a coordinate $z$ on $\CC^*$ identifies the coordinate ring of $\wt{LG} \times \CC^*$ with the set of regular maps $\wt{LG} \to \CC[z^{\pm 1}]$.  If our original function on $G$ is the character of a representation $V$, we refer to the resulting map $\wt{LG} \to \CC[z^{\pm 1}]$ as the \emph{evaluation character} of $V$.  The coefficient of any power of $z$ in an evaluation character is then an invariant scalar function on $\wt{LG}$.  

Together, all such coefficients of evaluation characters provide an infinite collection of pairwise Poisson-commuting functions on $\wt{LG}$.  Thus a natural strategy for constructing integrable systems is to restrict these functions to the double Bruhat cells of $\wt{LG}$.  On a general cell, however, it may be that too few of these functions remain independent to form a maximal set of Poisson-commuting functions.  Our main theorem provides a sufficient condition for obtaining an integrable system this way, or more precisely after reducing by the conjugation action of $H$.

\begin{thm}\label{thm:intmain}
The reduced Coxeter double Bruhat cell $\wt{LG}^{u,v}/H$ is the phase space of an integrable system whose Hamiltonians $H_1,\dots,H_{r+1}$ are coefficients of evaluation characters.  We take $H_1,\dots,H_r$ to be the constant coefficients of the evaluation characters of the $r$ fundamental representations of $G$, and $H_{r+1}$ to be the $z$-linear coefficient of the evaluation character of a certain representation $V$.  This is the irreducible representation whose highest weight is in the $W$-orbit of $\mu := - \sum_{k\neq 0}(\th_k + \sum_{j >_\si k}\th_j C_{kj})\om_k$, where the $\om_k$ are the fundamental dominant weights of $G$ and $\th_0=1$.  
\end{thm}

Note that in the statement of the theorem we could have taken $V$ to be any sufficiently large representation.  The given choice is essentially the minimal possible choice to ensure that $H_{r+1}$ restricts nontrivially to $\wt{LG}^{u,v}/H$. 

\begin{proof}\
By \cref{prop:leafdim} the symplectic leaves of $\wt{LG}^{u,v}/H$ are $(2r+2)$-dimensional, so the stated functions will form an integrable system once we show that their Hamiltonian vector fields remain independent when restricted to $\wt{LG}^{u,v}/H$.  Since $\wt{LG}_\mb{i}$ is dense in $\wt{LG}^{u,v}$ it suffices to consider their restrictions to $\wt{LG}_\mb{i}/H$, where we can use the explicit coordinates given by \cref{eq:coords}.  

First we show that $H_{r+1}$ is nonzero when restricted to $\wt{LG}^{u,v}/H$.  We can compute the evaluation character of $V$ by decomposing the action of $g$ with respect to a weight basis.  Specifically, let $V_\la$ be the $\la$-weight space of $V$, $\pi_\la$ the projection of $V$ onto $V_\la$ given by the weight space decomposition, and $H_\la$ the regular function defined by $H_\la(g) := \tr_{V_\la}(\pi_\la \circ g)$.  Then $H_{r+1} = \sum H_{\la}$, where the sum runs over the nonzero weight spaces of $V$.

Recall that for any $g \in \wt{LG}_\mb{i}$ we have the factorization
\begin{gather}\label{eq:gfact}
g = \phi_{\si(0)}(g_{\si(0)}) \dots \phi_{\si(r)}(g_{\si(r)}) \phi_{\tau(0)}(g'_{\tau(0)}) \dots \phi_{\tau(r)}(g_{\tau(r)}'),
\end{gather}
where
\begin{equation*}
g_i = \begin{pmatrix} A_i & B_i \\ 0 & A_i^{-1} \end{pmatrix}, \quad
g'_i = \begin{pmatrix} D_i^{-1} & 0 \\ C_i & D_i \end{pmatrix}.
\end{equation*}
From \cref{lem:mu} we conclude that the weight spaces in $V$ of weight $\mu + \sum_{k \geq j}\th_{\si(k)}\al_{\si(k)}$ are nonzero for all $j$.  From this and \cref{eq:gfact} we see that for any $v \in V_\mu$, the component of $\phi_{\si(j)}(g_j) \dots$ $\dots \phi_{\si(r)}(g_r) \cdot v$ of weight $\mu + \sum_{k \geq j}\th_{\si(k)}\al_{\si(k)}$ is nonzero for all $j$.  Since $s_{\si(0)}\dots s_{\si(r)}(\mu)=\mu$, it follows that the $z$-linear term of $H_\mu$ contains a monomial whose $B_i$ components are exactly $B_0(\prod_{i\neq0}B_i^{\th_i})$.  One can compute from the weight spaces involved that this monomial does not depend on the $A_i$.  By inspecting the generators of $\CC[\wt{LG}_{\mb{i}}/H]$ from \cref{eq:coords} we conclude that this monomial must be a scalar multiple of $Q$.  In particular $H_\mu$ can be written as a sum of scalar multiple of $Qz$ and other terms not of this form.  The reader may check using \cref{eq:gfact} that $H_\la$ cannot contain any scalar multiple of $Qz$ unless $\la = \mu$.  In particular, the $z$-linear term of the evaluation character is nonzero, since we have ruled out any cancellation of the $Qz$.

The independence of $H_{r+1}$ and the remaining Hamiltonians follows from the fact that the restriction of $H_{r+1}$ to $\wt{LG}_{\mb{i}}/H$ is linear in $Q$, while the other Hamiltonians do not depend on $Q$.  Indeed, suppose $M$ is any monomial in the restriction of an evaluation character to $\wt{LG}_{\mb{i}}/H$.  It is straightforward to see that the power of $z$ accompanying $M$ is the difference of the exponents of $B_0$ and $C_0$ in $M$.  Since $Q$ is the only generator of $\CC[\wt{LG}_{\mb{i}}/H]$ whose powers of $B_0$ and $C_0$ are distinct, it follows that the $z^k$-term of an evaluation character has degree $k$ with respect to $Q$.  

Finally, we claim that the Hamiltonians $H_1,\dots,H_r$ are algebraically independent.  Decompose each $H_i$ as $J_i + K_i$, where $J_i$ has degree zero with respect to the $S_i$, and $K_i$ is a sum of monomials of nonzero degree in the $S_i$.  Since $H_i$ is the restriction of a function on $\wt{LG}$, $\lim_{B_j,C_j \to 0} H_i$ exists for all $j$, so these monomials are in fact of positive degree in the $S_i$.

We claim that the $J_i$ are independent.  The projection $\wt{H} \to H$ induces an inclusion $\CC[H] \subset \CC[\wt{H}]$, and we identify $\CC[\wt{H}]$ with $\CC[T^{\pm 1}_0,\dots ,T^{\pm 1}_r]$ in the obvious way.  Then restricting the characters of the $i$ fundamental representations to $H$ and including them in $\CC[T^{\pm 1}_0,\dots ,T^{\pm 1}_r]$, we obtain exactly the functions $J_i$; it is a standard result that the restrictions of the fundamental characters to $H$ are independent.

Now suppose there is some polynomial relation among the $H_i$.  That is, for some polynomial $p$ in $r$ variables we have $p(H_1,\dots,H_r)=0$.  For any polynomial $p$ we can consider the decomposition of $p(H_1,\dots,H_r)$ into a component of degree zero in the $S_i$ and a component which depends nontrivially on the $S_i$.  But the $K_i$ are all of strictly positive degree in the $S_i$, hence the degree zero part of $p(H_1,\dots,H_r)$ is exactly $p(J_1,\dots,J_r)$.  Thus $p(H_1,\dots,H_r)=0$ implies $p(J_1,\dots,J_r)=0$, so $p$ must be identically zero.  Finally, one can check using \cref{eq:explicitbrackets,prop:leafdim} that for the Hamiltonians $H_1,\dots,H_{r+1}$, their algebraic independence implies the generic independence of their Hamiltonian vector fields.
\end{proof}

\begin{lemma}\label{lem:mu}
We have $s_{\si(j)}\dots s_{\si(r)}(\mu)=\mu + \sum_{k \geq j}\th_{\si(k)}\al_{\si(k)}$ for all $j$.  Here $s_0$, $\al_0$ are understood as $s_{\th}$, $-\th$ rather than affine simple roots. In particular, $s_{\si(0)}\dots s_{\si(r)}(\mu)=\mu$, since $\th_0 \al_0 = -\sum_{i\neq 0} \th_i \al_i$.
\end{lemma}
\begin{proof}[Proof of \cref{lem:mu}]
We induct on $j$: assuming the statement for $j+1$ we compute that
\begin{align*}
s_{\si(j)}\dots s_{\si(r)}(\mu) &= s_{\si(j)}(\mu + \sum_{k > j}\th_{\si(k)}\al_{\si(k)}) \\
&= (\mu + \sum_{k > j}\th_{\si(k)}\al_{\si(k)}) - \langle \mu + \sum_{k > j}\th_{\si(k)}\al_{\si(k)} | h_{\si(j)} \rangle \al_{\si(j)} \\
&= \mu + \sum_{k \geq j}\th_{\si(k)}\al_{\si(k)} \\
\end{align*}
For $\si(j) \neq 0$ the last equality follows from the definition of $\mu$, while for $\si(j) = 0$ it follows from calculating that:
\begin{align*}
\langle \mu + \sum_{k >_{\si} 0}\th_k\al_k | h_0 \rangle &= \langle \mu + \sum_{k >_{\si} 0}\th_k\al_k | - \sum_{k \neq 0} d_k \th_k h_k \rangle \\
&= \sum_{k \neq 0}d_k\th_k(\th_k + \sum_{j >_\si k}\th_j C_{kj}) - \sum_{\substack{k \neq 0 \\ j >_\si 0}}d_k \th_k \th_j C_{kj} \\
&= \sum_{k \neq 0}d_k\th_k(\th_k + \sum_{\substack{j \neq 0 \\ j >_\si k}}\th_j C_{kj}) + \sum_{k <_\si 0}d_k\th_k C_{k0} - \sum_{\substack{k \neq 0 \\ j >_\si 0}}d_k \th_k \th_j C_{kj} \\
&= \frac12\sum_{j,k \neq 0}d_k \th_j \th_k C_{kj} - \sum_{\substack{j \neq 0 \\ k <_\si 0}}d_k\th_j \th_k C_{kj} - \sum_{\substack{k \neq 0 \\ j >_\si 0}}d_k \th_k \th_j C_{kj} \\
& = -1. \\
\end{align*}
Here we use the fact that $\sum_{j,k \neq 0}d_k \th_j \th_k C_{kj} = \langle \th | h_\th \rangle = 2$, $C_{k0} = - \sum_{j\neq0}\th_j C_{kj}$, and $C_{kk} = 2$.
\end{proof}

\begin{remark}
Even for double Bruhat cells on which there are too few independent coefficient functions to obtain an integrable system, it was shown in \cite{Reshetikhin2003} that in the finite-dimensional case one obtains superintegrable systems.  This is a stronger statement than simply having a collection of Poisson-commuting functions.  In particular, the dynamics are restricted to isotropic analogues of Liouville tori.  One expects this to hold in the affine case as well, but we do not pursue this here.
\end{remark} 

\subsection{The Relativistic Periodic Toda System}\label{sec:toda}

We now show that the relativistic periodic Toda system of \cite{Ruijsenaars1990} can be realized (up to symplectic reduction) as an affine Coxeter-Toda system of type $A_n^{(1)}$ for a natural choice of Coxeter elements.  In canonical coordinates $p_k, q_k$ this system corresponds to the Hamiltonian
\begin{equation}\label{eqn:rtshamiltonian}
\sum_{k=0}^m e^{hp_k}(1 + h^2 exp(q_{k+1} - q_k)),
\end{equation}
where $h$ is a nonzero parameter and we impose the periodic boundary conditions $p_{k+m+1}=p_k$, $q_{k+m+1}=q_k$ \cite{Suris1991}.  For now we consider the complex form where $p_k$ and $q_k$ take values in $\CC$.

Consider the double Bruhat cell of $\wt{LSL}_n$ with $u$ and $v$ both equal to the element $s_0 s_1 \cdots s_n$, where the simple roots of $SL_n$ are numbered in the usual way.  We note that from the computations in \cref{sec:brackets} it follows that the symplectic leaves of this cell are already $(2r+2)$-dimensional, so the corresponding Coxeter-Toda system is integrable before reduction by $H$.  

If $H_1 \in \CC[(\wt{LSL}_n)_{\mb{i}}]$ is the Hamiltonian obtained from the constant term of the character of the defining representation of $SL_n$, a simple calculation yields that
\begin{equation}\label{eqn:actshamiltonian}
H_1 =  \sum_{i=0}^n T_i T_{i-1}^{-1}(1 + S_i),
\end{equation}
where $T_{-1}$ and $S_{-1}$ are read as $T_n$ and $S_n$.

To connect this with the relativistic Toda system, we introduce auxiliary variables ${c_0, \dots, c_n, d_0, \dots, d_n}$, on which we define a Poisson structure by setting
\[
\{c_k, d_k\} = 2 c_k d_k, \quad \{c_k, d_{k+1}\} = -2 c_k d_{k+1}, \quad \{c_k, c_{k+1} \} = -2 c_k c_{k+1},
\]
with all other brackets among the generators equal to zero (here $d_{n+1}$ and $c_{n+1}$ are understood as $d_0$ and $c_0$).  The algebra $\CC[c_0^{\pm 1}, d_0^{\pm 1},\dots,c_n^{\pm 1}, d_n^{\pm 1}]$ is then the coordinate ring of a $(2n+2)$-dimensional Poisson torus with $2n$-dimensional symplectic leaves.  

Now observe that this Poisson variety can be obtained as a reduction of both $(\wt{LSL}_n)_{\mb{i}}$ and the phase space of the relativistic Toda system (for $m = n$ and $h=2$). That is, we have surjective Poisson maps given by
\[c_i \mapsto S_i T_i T_{i-1}^{-1}, \quad d_i \mapsto T_i T_{i-1}^{-1} \quad \mathrm{and} \quad 
c_i \mapsto 4e^{2p_i - q_i + q_{i+1}}, \quad d_i \mapsto e^{2p_i}.
\]
Moreover, the following proposition is clear from \cref{eqn:rtshamiltonian,eqn:actshamiltonian}:
\begin{prop} The Hamiltonian
\[
H_1 = \sum_{i=0}^n c_i + d_i
\]
pulls back to the Hamiltonians of the relativistic Toda and Coxeter-Toda systems under the maps given above, hence defines a Hamiltonian system which is a common reduction of these two integrable systems.
\end{prop}

Finally, we recall that the relativistic Toda system is usually defined on the real phase space with canonical coordinates $p_k$, $q_k$.  Because of the exponentials in the Hamiltonian, the corresponding real slice of the Coxeter-Toda phase space is the subset of $(\wt{LSL}_n)_{\mb{i}}$ on which the factorization coordinates take \emph{positive} real values.  This totally positive part of the double Bruhat cell has many interesting combinatorial properties and was the principal motivation for \cite{Fomin1999}.  Thus in the present context we find that total positivity arises naturally when we compare our construction with the usual real form of the relativistic Toda system.

\bibliography{DBCIS_arxiv} 
\bibliographystyle{alpha}

\end{document}